\documentclass[journal]{IEEEtran}
\usepackage{amsfonts}
\usepackage{multirow}

\newcommand{\BE}{\mathbf{E}}
\newcommand{\BH}{\mathbf{H}}
\newcommand{\BJ}{\mathbf{J}}
\newcommand{\BV}{\mathbf{V}}
\newcommand{\Bv}{\mathbf{v}}
\newcommand{\Bs}{\mathbf{s}}
\newcommand{\Bn}{\mathbf{n}}

\newcommand{\Bnu}{\boldsymbol{\nu}}
\newcommand{\Bzero}{\boldsymbol{0}}

\newcommand{\Bphi}{\boldsymbol{\phi}}
\newcommand{\Bpsi}{\boldsymbol{\psi}}

\newcommand{\WBE}{\widetilde{\mathbf{E}}}
\newcommand{\WBH}{\widetilde{\mathbf{H}}}

\newcommand{\WE}{\widetilde{E}}
\newcommand{\WH}{\widetilde{H}}

\newcommand{\womega}{\widetilde{\omega}}
\newcommand{\wlambda}{\widetilde{\lambda}}

\newcommand{\bnu}{\bar{\nu}}

\newcommand{\Bbnu}{\bar{\boldsymbol{\nu}}}

\newcommand{\eps}{\epsilon}

\newcommand{\Dx}{\Delta x}
\newcommand{\Dy}{\Delta y}
\newcommand{\Dz}{\Delta z}
\newcommand{\Dt}{\Delta t}
\newcommand{\half}{\frac{1}{2}}

\usepackage{amsthm}
\newtheorem{thm}{Theorem}
\newtheorem{rem}{Remark}

\newcommand\pzc[1]{\textcolor{black}{#1}}

 \usepackage[caption=false,font=footnotesize]{subfig}

\usepackage{graphicx,xcolor,comment,ulem,url}

\usepackage[linesnumbered,ruled,vlined]{algorithm2e}
\usepackage{algpseudocode}

\newcommand{\revtwo}[2]{{\color{black}{#1}}{}}
\newcommand{\revtwomath}[1]{{\color{black}{#1}}}

\ifCLASSINFOpdf
\else
\fi
\usepackage{amsmath}
\begin{document}
%
\title{EM-WaveHoltz: A flexible frequency-domain method built from time-domain solvers}
%
%
%

\author{Zhichao~Peng
        and~Daniel~Appel\"{o}
\thanks{Zhichao Peng is with the Department of Mathematics, Michigan State University, East Lansing, MI 48824.}
\thanks{Daniel~Appel\"{o} is with the 
Department of Computational Mathematics, Science \& Engineering and the 
Department of Mathematics, Michigan State University, East Lansing, MI 48824
}
\thanks{Manuscript received \today; revised \today.}}

%
%

\markboth{ IEEE TRANSACTIONS ON ANTENNAS AND PROPAGATION}%
{Shell \MakeLowercase{\textit{et al.}}: Bare Demo of IEEEtran.cls for IEEE Journals}
%



\maketitle

\begin{abstract}
 A novel approach to computing time-harmonic solutions of Maxwell's equations by time-domain simulations is presented. The method, {EM}-{W}ave{H}oltz, results in a positive definite system of equations which makes it amenable to iterative solution with the conjugate gradient method or with GMRES. Theoretical results guaranteeing the convergence of the method away from resonances are presented. Numerical examples illustrating the properties of {EM}-{W}ave{H}oltz are given.  
\end{abstract}


%

%
%
%
%
\IEEEPARstart{T}{WO} of the main challenges when solving the time-harmonic Maxwell equations at high frequencies are the indefinite nature of the Maxwell system and the high resolution requirement. Without proper preconditioners, iterative solvers such as GMRES and BICG may converge slowly. These challenges are similar to the ones for solving the Helmholtz equation  at high frequencies. Recently, we introduced a  scalable iterative method called WaveHoltz \cite{appelo2020waveholtz} for the Helmholtz equation. In this paper, we introduce the electromagnetic-WaveHoltz (EM-WaveHoltz) method, which can be seen as a generalization of the WaveHoltz method to the time-harmonic (or frequency-domain) Maxwell equations. The proposed EM-WaveHoltz method converts the frequency-domain problem to a fix point problem in the time-domain. The fixed point iteration is linear and can be rewritten as a linear system of equations with a system matrix that is positive definite and that can therefore be efficiently inverted using standard Krylov methods such as GMRES.   

In the EM-WaveHoltz method, we convert the frequency-domain problem to a time-domain problem by evolving and filtering Maxwell's equations with periodic forcing over one time period. When applied, this filter results in the time-domain solution converging to a fix point where the solution becomes equivalent to the solution of the frequency-domain problem. Salient features of the EM-WaveHoltz method are as follows. 
\begin{enumerate}
\item The resulting linear system is always positive definite (sometimes symmetric). 
\item The EM-WaveHoltz method can be driven by any scalable  time-domain solver, for example the finite difference time-domain (FDTD) method  \cite{Taflove:2005fk} and discontinuous Galerkin time-domain method (DGTD) \cite{HesthavenWarburton02}.  
\item A unique feature of the EM-WaveHoltz method is that it is possible to obtain frequency-domain solutions for multiple frequencies at once but at the cost of a single solve.

\end{enumerate}

We note that properties of our method are to some extent shared with the properties of the controllability method. In particular the controllability method finds the solution to the frequency-domain problem by using time-domain solvers like our approach. However, while our formulation relies on a fixed point iteration, the controllability method seeks to minimize the deviation from time-periodicity of the initial and final data of the time-domain simulation. The controllability method was first proposed for a time-harmonic wave scattering problem  \cite{bristeau1998controllability}, and we refer readers to \cite{grote2019controllability} for recent \pzc{development}. The controllability method is also generalized to the time-harmonic Maxwell equation in second order formulation \cite{bristeau19993d} and the first \pzc{order} formulation \cite{pauly2011theoretical,rabina2014comparison}.  One main difference between our method and the controllability method is that the controllability method needs backward solves, while our method does not.

There are of course many other methods that have been designed for efficiently solving the frequency-domain Maxwell's equations. For scattering and radiation problems in homogenous media integral equation formulations are known to be highly efficient and yield fast algorithms \cite{chew2001fast,andriulli2008multiplicative}. Domain decomposition methods (DDM) \cite{toselli2006domain} have also achieved success for the time-harmonic electromagnetic problems \cite{lee2005non,vouvakis2006fem,peng2010one,peng2010non,dolean2015effective}. The DDM method and the integral equation method have been combined in \cite{peng2011integral}. Recently, \cite{bonazzoli2019domain} extends the ``shifted-Laplacian preconditioner" for the Helmholtz equation to the high frequency time-harmonic Maxwell equations and designs an optimal DDM method. Multigrid methods  have also been considered for the time-harmonic Maxwell equations \cite{hiptmair1998multigrid,gopalakrishnan2004analysis}. A multigrid method for the high frequency time-harmonic Maxwell equations is designed in \cite{lu2016robust}. Sweeping preconditioners for time-harmonic Maxwell equations, which utilize the intrinsic structure of the Green's function, have been developed for  the Yee scheme \cite{tsuji2012sweepingyee} and the finite element method \cite{tsuji2012sweepingfem}. We finally note that it also possible to directly use a time-domain solver in other ways to find the frequency domain solution. The most straightforward approach is to save the solution for some time $T$ and then take a Fourier transform. The upside with this approach is that it produces an approximate result to the frequency domain problem for many frequencies at once. The drawbacks are that the solution is approximate with an accuracy that\revtwo{, in the case of a continuous wave sinusoidal source, scales as}{depends on} $T^{-1}$ and that the need to save the solution makes this approach memory intensive.  \revtwo{This slow convergence can be improved if modulated sources are used and the resulting method can be more efficient than a single frequency solver if the frequency response is desired over a broad spectrum of frequencies and the accuracy requirements are less stringent.}{} For open problems it is possible to appeal to the limiting amplitude principle  \cite{morawetz1962limiting} and simply let a harmonically forced problem converge to the frequency domain problem by simulating long enough. The convergence of this approach is severely impacted when trapping geometry is present and the principle is not valid for closed domains. \revtwo{Also for Fourier transformed methods the required simulation time becomes prohibitive when closed domains with the quality factor $Q=\infty$ are considered.}{}

The rest of  this paper is organized as follows. In Section \ref{sec:continuous_waveholtz}, we present the EM-WaveHoltz formulation for the continuous equations and discuss the properties of the resulting linear system, the choice of the linear solver, and present how to obtain solutions for multiple frequencies in one solve. In Section \ref{sec:discerete_waveholtz}, to show the flexibility with respect to the choice of the time-domain solvers, we couple the EM-WaveHoltz method, first with the Yee scheme and then with the discontinuous Galerkin (DG) method. In Section \ref{sec:numerical}, the performance of the EM-WaveHoltz method is demonstrated through a series of numerical examples. \revtwo{A simple implementation of the method in 1D to aid the reader in understanding the details of the method can be found at \url{https://zhichaopengmath.github.io/code/}.}{}

\section{Electromagnetic WaveHoltz iteration for the Maxwell's equation}\label{sec:continuous_waveholtz}
We consider the frequency-domain Maxwell's equation:
\begin{subequations}
\label{eq:maxwell_frequency}
\begin{align}
&i\omega \eps \BE = \nabla\times \BH-\BJ,\\
&i\omega \mu \BH = -\nabla\times \BE,
\end{align}
\end{subequations}
closed by boundary conditions corresponding to either a perfect electric conductor or to an unbounded domain. Here $\BE$ and $\BH$ are the complex valued electric and magnetic fields, \pzc{$\eps$, $\mu$ are real valued permittivity and permeability} and $J$ is the real valued current source. Taking the real and imaginary parts we find 
\begin{subequations}
\label{eq:maxwell_frequency_re_im}
\begin{align}
-\omega \Im \{ \eps \BE \} &= \Re \{ \nabla\times \BH \} -\BJ, \label{eq:maxwell_frequency_re_im_a}\\
 \omega \Re \{ \eps \BE \} &= \Im \{ \nabla\times \BH \},\label{eq:maxwell_frequency_re_im_b} \\
-\omega \Im \{ \mu \BH\} &= -\Re \{\nabla\times \BE \}, \label{eq:maxwell_frequency_re_im_c} \\
\omega \Re \{ \mu \BH\} &= -\Im \{\nabla\times \BE \}. \label{eq:maxwell_frequency_re_im_d}
\end{align}
\end{subequations}

We want to relate the fields $\BE$ and $\BH$ to  real valued and $T = 2\pi/ \omega$-periodic \revtwo{fields}{solutions} 
\begin{subequations}
\label{eq:periodic_ansatz}
\begin{align}
\WBE &= \hat{{\bf E}}_0 \cos(\omega t) + \hat{{\bf E}}_1 \sin(\omega t), \\
\WBH &= \hat{\bf{H}}_0 \cos (\omega t) + \hat{\bf{H}}_1 \sin (\omega t),
\end{align}
\end{subequations}
\revtwo{that are solutions of the time-domain equations}{to the time-domain equations} 
\begin{subequations}
\label{eq:maxwell_time}
\begin{align}
&\eps \partial_t  \WBE =\nabla\times \WBH- \sin(\omega t)\BJ,\\
&\mu  \partial_t \WBH = -\nabla\times \WBE.
\end{align}
\end{subequations}
For such periodic solutions we can match the $\sin(\omega t)$ and $\cos(\omega t)$ terms to find the relations
\begin{subequations} 
\label{eq:maxwell_time_details}
\begin{align}
\pzc{-\omega ( \eps \hat{{\bf E}}_0 )}  &= \nabla\times  \hat{\bf{H}}_1 - \BJ, \label{eq:apa_a} \\
 \pzc{\omega (\eps\hat{{\bf E}}_1 ) } &=  \nabla\times  \hat{\bf{H}}_0 , \label{eq:apa_b}\\
\pzc{- \omega (\mu \hat{{\bf H}}_0 )}  &= - \nabla\times  \hat{\bf{E}}_1, \label{eq:apa_c} \\
\pzc{ \omega (\mu\hat{{\bf H}}_1 ) } &= -\nabla\times  \hat{\bf{E}}_0. \label{eq:apa_d}
\end{align}
\end{subequations}
\revtwo{Comparing (\ref{eq:maxwell_frequency_re_im_a}) with (\ref{eq:apa_a}) and (\ref{eq:maxwell_frequency_re_im_c}) with (\ref{eq:apa_c}), }{} it now follows that the initial data of $\WBE$ and $\WBH$ \revtwo{matches}{determines} the imaginary part of the frequency-domain solution 
\begin{align*}
\Im \{ \BE \} = \hat{{\bf E}}_0, \ \
 \Im \{ \BH \}= \hat{{\bf H}}_0.  
\end{align*}
\revtwo{Also, from (\ref{eq:maxwell_frequency_re_im_b}), (\ref{eq:apa_b}) and (\ref{eq:maxwell_frequency_re_im_d}), (\ref{eq:apa_d}), we get }{as well as the real part}
\begin{align} \label{eq:get_real}
\Re\{ \BE \}  = \hat{{\bf E}}_1 = \frac{1}{\eps} \nabla\times  \hat{\bf{H}}_0,&&
\Re \{ \BH \} = \hat{{\bf H}}_1 = -\frac{1}{\mu} \nabla\times  \hat{\bf{E}}_0.
\end{align}

Our EM-WaveHoltz method finds the periodic solutions \eqref{eq:periodic_ansatz} by iteratively determining the initial data 
to \eqref{eq:maxwell_time}.

 Define the filtering operator, $\Pi$,  acting on the initial conditions $\Bnu=(\Bnu_E,\Bnu_H)^T$:
\begin{align}
\label{eq:filter}
\Pi\Bnu=
\Pi \left(\begin{matrix}
	    \Bnu_E\\
	    \Bnu_H
	     \end{matrix}\right)
=\frac{2}{T}\int^{T}_{0}\left(\cos(\omega t)-\frac{1}{4}\right)
\left(\begin{matrix}
	    \revtwomath{\WBE_{\Bnu}}\\
	    \revtwomath{\WBH_{\Bnu}}
	\end{matrix}\right)
	dt,
\end{align}
with $T=2\pi / \omega$ and \revtwo{$\WBE_{\Bnu}$ and $\WBH_{\Bnu}$ being the fields resulting from the initial conditions $\Bnu=(\Bnu_E,\Bnu_H)^T$.}{}

\revtwo{By construction $\Pi (\Im\{\BE\},\Im\{\BH\})^T = (\Im\{\BE\},\Im\{\BH\})^T$, and as $(\Re\{\BE\},\Re\{\BH\})^T$ can be computed directly via \eqref{eq:get_real}, the solution to the frequency-domain equation is the fix-point of the operator $\Pi$.}{By construction, one can verify that} \revtwo{}{ $\Pi (\Im\{\BE\},\Im\{\BH\})^T = (\Im\{\BE\},\Im\{\BH\})^T$.} \revtwo{}{As the real part can be computed directly via \eqref{eq:get_real},} \revtwo{}{the solution to the frequency-domain equation is} \revtwo{}{a fix-point of the operator $\Pi$.} 

\revtwo{The operator $\Pi$ is contractive. Precisely, if a certain initial data gives rise to a solution that, in addition to the $\sin(\omega t )$ and $ \cos (\omega t)$ terms in (\ref{eq:periodic_ansatz}), has other time-harmonic components, e.g. $\sin(\omega^\prime t ), \, \omega^\prime \neq \omega$, then in each iteration the filter reduces the amplitude of those components.}{}

Based on these facts, we define the EM-WaveHoltz iteration: 
\begin{align}
\pzc{\Bnu^{n+1}}=\Pi \Bnu^n,\ \ \text{with}\ \ \Bnu^0=(\Bnu^0_E,\Bnu_H^0)^T=\Bzero.
\end{align}
The EM-WaveHoltz iteration converges to the imaginary parts of the solution to the frequency-domain equation 
\begin{align}
\lim_{n\rightarrow\infty}\Bnu^n=\lim_{n\rightarrow\infty}(\Bnu_E^n,\Bnu_H^n)^T=(\Im\{\BE\},\Im\{\BH\})^T,
\end{align}
and the real parts can be recovered via \eqref{eq:get_real}.
\pzc{
\begin{rem}
Alternatively we could formulate the time-domain problem with a cosine forcing
\begin{subequations}
\label{eq:maxwell_time_cos}
\begin{align}
&\eps \partial_t  \WBE =\nabla\times \WBH- \cos(\omega t)\BJ,\\
&\mu  \partial_t \WBH = -\nabla\times \WBE.
\end{align}
\end{subequations}
Again, the real valued $T= 2\pi/ \omega$-periodic solutions to \eqref{eq:maxwell_time_cos} are \revtwo{of}{on} the form \eqref{eq:periodic_ansatz} but (see Appendix \ref{sec:cos_forcing}) the solution to \eqref{eq:maxwell_time_cos} $\WBE$ and $\WBH$ have a slightly different relation to the frequency-domain solution
$$
\Re\{E\}=\hat{E}_0, \;\;\Re\{H\} = \hat{H}_0, \;\;\Im\{E\}=-\hat{E}_1, \;\;\Im\{H\} = -\hat{H}_1. 
$$
With the same filter and iteration process defined as the $\sin$-forcing case, we have
\begin{align}
\lim_{n\rightarrow\infty}\Bnu^n=\lim_{n\rightarrow\infty}(\Bnu_E^n,\Bnu_H^n)^T=(\Re\{\BE\},\Re\{\BH\})^T.
\end{align}
In our numerical tests, we find that the number of iterations needed by the EM-WaveHoltz method are essentially identical for the two alternatives. In this paper, we focus on the EM-Waveholtz with the $\sin$-forcing.
\end{rem}
}
\subsection{EM-WaveHoltz for the energy conserving case}
For \pzc{real-valued $\eps$, $\mu$ and $J$, with} PEC boundary conditions and other boundary conditions \revtwo{that}{} lead to a conservation of the electromagnetic energy in a bounded domain,  the EM-WaveHoltz iteration can be simplified further. For such problems, \revtwo{assuming that $\omega$ is not a resonance frequency of the cavity,
}{} $\Im \{\BH\}$ is identically zero and the EM-WaveHoltz iteration is reduced to 
\begin{align}
\Bnu_E^{n+1} = \Pi \Bnu_E^n, \quad\Bnu_H^n=\Bzero,\quad\Bnu_E^0=\Bzero,
\end{align}
where now
\begin{align}
\Pi\Bnu_E = \frac{2}{T}\int_{0}^T \left(\cos(\omega t)-\frac{1}{4}\right)\revtwomath{\WBE_{\Bnu}} dt.
\end{align}
As long as $\omega$ is not a resonance this simplified EM-WaveHoltz iteration converges  
\begin{align}
\Bnu_E^n=\Im\{\BE\},\;\text{as}\; n\rightarrow\infty.
\end{align}

\subsection{Krylov acceleration}
For unbounded problems where $\omega$ is close to a resonance or for bounded problems with trapping geometries, the convergence of the WaveHoltz fix point iteration can be slow \cite{appelo2020waveholtz}. Fortunately as the iteration is linear, it is easy to rewrite it as a positive definite linear operator that can be efficiently  inverted by a Krylov subspace method. To see this we introduce the operator:
\begin{align}
\label{eq:def_S_operator}
S\Bnu = \Pi \Bnu - \Pi \Bzero.
\end{align}
Then, based on the definition of $S$, we have
\begin{align}
\Pi \Bnu = S\Bnu+\Pi \Bzero.
\end{align}
Hence, finding the fix point of $\Pi$: $\Pi \Bnu =\Bnu$ is equivalent to solving the equation
$(I-S)\Bnu=\Pi \Bzero$. \pzc{Here, we want to emphasize that $\Pi\Bzero\neq \Bzero$ unless frequency-domain problem has zero solutions (see \eqref{eq:pi_zero_result} in Appendix B for more details). Here $\Bzero$ stands for the zero initial condition in the time-domain, and with a non-zero source, the filtered time-domain solution over one  period $\Pi\Bzero$ is very likely nonzero.}

A Krylov method such as the conjugate gradient method, GMRES or TFQMR can be applied to solve $(I-S)\Bnu = \Pi \Bzero$ in a matrix-free manner. In practice, to obtain the right hand side $\Pi \Bzero$, we just need to solve the time-domain problem \eqref{eq:maxwell_time} with zero initial conditions $\Bnu=\Bzero$ from $t=0$ to $t=T$ and use a numerical quadrature to approximate the filter $\Pi \Bzero$ as we march in time. To calculate the matrix multiplication $(I-S)\Bnu$, we can utilize the fact that 
\begin{align*}
(I-S)\Bnu = \Bnu - (\Pi\Bnu-\Pi \Bzero)=\Bnu-\Pi\Bnu+\Pi \Bzero.
\end{align*}
That is, for a given $\Bnu$ and $\Pi\Bzero$  precomputed, we just need to compute $\Pi\Bnu$ to obtain the action of $(I-S)$ onto $\Bnu$. Recall that $\Pi\Bnu$ is obtained by computing the filter by a numerical quadrature incrementally as the solution to \eqref{eq:maxwell_time} is evolved for one $T = 2\pi / \omega$ period with $\Bnu$ as the initial conditions. Thus the cost to compute one Krylov vector is that of a wave solve with one additional variable needed to sum up the projection throughout the evolution. 

When using GMRES there is always a concern about the size of the Krylov subspace as the number of iterations grow. Here our method has a significant upside to solving the frequency-domain problem. Note that although we are looking for a $T = \frac{2\pi}{\omega}$-periodic solution, there is nothing in the method that prevents us from changing the filtering to extend over a longer time, say,  $T=N\frac{2\pi}{\omega}$, with $N$ a positive integer. As we show in the numerical examples below, for moderate $N$ this reduces the number of iterations by a factor of $N$ so that the overall computational cost is the same. \pzc{For GMRES without restart, filtering over longer periods reduces the memory needed. For GMRES with restart, filtering over longer periods reduces the number of restart needed.}

In Appendix \ref{sec:analysis_continuous},  we show that $I-S$ is always a positive definite operator and for energy conserving boundary conditions it is also self-adjoint. These results carry over to the discretized equations in the sense that the matrix that needs to be inverted is always positive definite and, if a symmetric and energy conserving method (like the Yee scheme) is used, the matrix is also symmetric for energy conserving boundary conditions like PEC. For the SPD case our method becomes particularly efficient and memory lean as the conjugate gradient method can be used. 

\revtwomath{Now, we summarize how to implement the EM-WaveHoltz method given a time-domain solver and a GMRES iterative solver. The filtering is presented as Algorithm \ref{alg:pi} and Algorithm \ref{alg:gmres_waveholtz} describes the GMRES/Krylov acceleration. }

\begin{algorithm}[!h]
\caption{\revtwomath{Given initial data $\Bnu$ and a time-domain solver, compute $\Pi\Bnu$.} \label{alg:pi} }
\begin{algorithmic}[1]
\State \revtwomath{Set $\Bnu=(\Bnu_E^T,\Bnu_H^T)$ as the initial condition for the time-domain solver.}
\State \revtwomath{Use the time domain solver to evolve the time-domain equation \eqref{eq:maxwell_time_cos} for one/multiple periods. In each time step, incrementally compute $\Pi_h\Bnu$ by the trapezoid rule for numerical integration.}
\State \revtwomath{After the solution has been evolved for one period in time return $\Pi_h\Bnu$.}
\end{algorithmic}
\end{algorithm}

\begin{algorithm}[!h]
\caption{\revtwomath{GMRES accelerated EM-WaveHoltz iteration.} \label{alg:gmres_waveholtz} }
\begin{algorithmic}[1]
\State{ \revtwomath{Compute $\Pi\mathbf{0}$ by Algorithm \ref{alg:pi}.} }
\Procedure{\revtwomath{MatMul}}{$\Bnu$} \Comment{Compute $(I-S)\Bnu$.}
	\State \revtwomath{Use Algorithm \ref{alg:pi} to compute $\Pi\Bnu$.}
	\State \revtwomath{\textbf{return} $(I-S)\Bnu=\Bnu-\Pi\Bnu+\Pi\mathbf{0}$.}
\EndProcedure
\Procedure{\revtwomath{Solve}}{{\tt TOL}} \Comment{Solve $(I-S)\Bnu = \Pi\mathbf{0}$}
	  \State \revtwomath{Set the initial guess as $\Bnu^0=(\mathbf{0},\mathbf{0})^T$.}
	  \State \revtwomath{Apply GMRES with matrix free $\textsc{MatMul}$ procedure. Stop if the relative residual is smaller than {\tt TOL}}.
\EndProcedure
\State \revtwomath{The solution produced by $\textsc{Solve}$ is the imaginary part of the frequency-domain solution.}
\end{algorithmic}
\end{algorithm}


\subsection{Multiple frequencies in one solve}\label{sec:multiple_frequency}
Similar to the WaveHoltz method for the Helmholtz  equation \cite{appelo2020waveholtz}, the EM-WaveHoltz method can be applied to obtain the solutions for multiple frequencies in one solve. 

Precisely, let $\omega_k=n_k\omega_0$, $k=1,\dots,N$ for some $\omega_0>0$ and $n_1<n_2<\dots<n_N$ being positive integers. Then in a traditional frequency-domain solver each frequency requires the solution of $N$ different systems
\begin{subequations}
\label{eq:maxwell_multiple_frequency}
\begin{align}
&i\omega_k \eps \BE_k = \nabla\times \BH_k-\BJ_k,\\
&i\omega_k \mu \BH_k = - \nabla\times \BE_k.
\end{align}
\end{subequations}

Now, assuming that each frequency solve has the same type of boundary condition and material properties  (the forcing $\BJ_k$ can be different for each $k$), we can solve for all frequencies at once. We take the energy conserving case as an example, then the {\bf single} time-domain problem we must solve is
\begin{subequations}
\label{eq:maxwell_multiple_frequency_time}
\begin{align}
&\eps \partial_t \WBE = \nabla\times\WBH- \sum_{k=1}^{N}\sin(\omega_k t)\BJ_k,\\
&\mu \partial_t \WBH = \nabla\times\WBE.
\end{align}
\end{subequations}
The converged solution to \eqref{eq:maxwell_multiple_frequency_time} can be decomposed as
\begin{align} \label{eq:all_at_once_sum}
\WBE = \sum_{k=1}^N \hat{\BE}_{k,0}\cos(\omega_k t),
\end{align}
where $\hat{\BE}_{k,0} = \Im \{\BE_k\}$ gives the solution to the original frequency-domain problem \eqref{eq:maxwell_multiple_frequency} corresponding to $\omega_k$. To obtain the ``all $k$'' solution  through EM-WaveHoltz is easy, the filtering operator simply needs to be modified as 
\begin{align}
\Pi\Bnu_E = \frac{2}{T}\int_{0}^T \left(\sum_{k=1}^{N}\cos(\omega_k t)-\frac{1}{4}\right)\revtwomath{\WBE_{\Bnu}} dt.
\end{align}
Here, the final time $T$ is chosen such that $T/(\frac{2\pi}{\omega_k})$ is an integer for all $k$. 

Once the EM-WaveHoltz iteration has converged to \eqref{eq:all_at_once_sum} we separate the different solutions by evolving \eqref{eq:maxwell_multiple_frequency} for one more $T_k$-period while applying the filters  
\begin{align}\label{eq:multiple_freqeuncy_filter}
\Im(\BE_k) &= \frac{2}{T_k} \int_{0}^{T_k} \left(\cos(\omega_k t)-\frac{1}{4}\right)\WBE dt, \\
\Re(\BE_k) &= \frac{2}{T_k} \int_{0}^{T_k} \sin(\omega_k t) \, \WBE dt.
\end{align}

\section{Discretization of the  EM-WaveHoltz method}\label{sec:discerete_waveholtz}
We have presented how the EM-WaveHoltz iteration converts a frequency-domain problem to a time-domain problem. In this section, we will use the Yee scheme \cite{yee1966numerical,taflove2005computational} and the discontinuous Galerkin (DG) method \cite{HesthavenWarburton02,hesthaven2007nodal,cockburn2012discontinuous} as examples of integrating the EM-WaveHoltz iteration in existing time-domain solvers. We also want to point out that it is possible to couple the EM-WaveHoltz method to other type time-domain solvers such as spectral element method and continuous finite element method. Further, although we don't consider it here, our approach directly generalizes to linear dispersive frequency-domain models such as the generalized dispersive materials modeled through an auxiliary differential equation approach in \cite{banks2020high}.  

\subsection{Yee-EM-WaveHoltz}
The Yee scheme \cite{yee1966numerical,taflove2005computational} or the finite-difference-time-domain (FDTD) method, is one of the most popular and successful methods in computational electromagnetics and can be easily turned into a fast FDFD method, the Yee-EM-WaveHoltz method, as follows.

\pzc{For brevity we consider the two dimensional TM model, then  $E_x=E_y=H_z=0$. Assume a uniform time step size $\Dt = T / M$ and denote a grid function at a point $(i\Dx,j\Dy,n\Dt)$ by $\revtwomath{F^n_{i,j}}$ and denote $t^n= n\Dt$.  
Then the Yee scheme to solve the time-domain problem in the EM-WaveHoltz formulation is: 
\begin{subequations}
\label{eq:yee_2d_tm}
\begin{align}
&\eps_{i,j}\frac{(\WE_z)^{n+1}_{i,j}-(\WE_z)^{n}_{i,j}}{\Dt}=\Big(\frac{(\WH_y)^{n+\half}_{i+\half,j}-(\WH_y)^{n+\half}_{i-\half,j}}{\Dx}\notag\\
&-\frac{(\WH_x)^{n+\half}_{i,j+\half}-(\WH_x)^{n+\half}_{i,j-\half}}{\Dy}\Big)-\sin(\omega t^{n+\frac{1}{2}}) (J_z)_{i,j}\label{eq:yee_2d_tm_E}\\
&\frac{(\WH_x)^{n+\half}_{i,j+\half}-(\WH_x)^{n-\half}_{i,j+\half}}{\Dt}
=-\frac{1}{\mu_{i,j+\half}}\frac{(\WE_z)^n_{i,j+1}-(\WE_z)^n_{i,j}}{\Dy},\label{eq:yee_2d_tm_Hx}\\
&\frac{(\WH_y)^{n+\half}_{i+\half,j}-(\WH_y)^{n-\half}_{i+\half,j}}{\Dt}
=\frac{1}{\mu_{i+\half,j}}\frac{(\WE_z)^{n}_{i+1,j}-(\WE_z)^n_{i,j}}{\Dx},\label{eq:yee_2d_tm_Hy}
\end{align}
\end{subequations}
}
For the initial step $\WH_x^{-\frac{1}{2}}$ is initialized as
\begin{align}
\label{eq:H_negative_half}
&\revtwomath{(\WH_x)^{-\half}_{i,j+\half}}=(\WH_x)^{0}_{i,j+\half}-\frac{\Dt}{2\revtwomath{\mu_{i,j+\half}}}\Big(-\frac{\revtwomath{(\WE_z)^0_{i,j+1}-(\WE_z)^0_{i,j}}}{\Dy}
\Big),
\end{align}
and $\WH^{-\half}_y$ \pzc{is} initialized similarly.

To approximate the filter operator $\Pi$ in \revtwo{\eqref{eq:filter}}{\eqref{eq:filterYEE}} we use the composite trapezoidal rule 
\begin{align}
\label{eq:filterYEE}
\Pi_h \left(\begin{matrix}
	    \Bnu_E\\
	    \Bnu_H
	     \end{matrix}\right)
=\frac{2\Dt}{T}\sum_{n=0}^M\eta_n\left(\cos(\omega t^n)-\frac{1}{4}\right)
\left(\begin{matrix}
	    \revtwomath{\WBE_{\Bnu}}^n\\
	  \revtwomath{\frac{\WBH_{\Bnu}^{n+\half}+\WBH_{\Bnu}^{n-\half}}{2}}
	\end{matrix}\right),
\end{align}
where 
\begin{align}
\eta_n = \begin{cases}
\frac{1}{2}, \; n=0\;\text{or}\;M,\\
1,\;\text{otherwise}.
\end{cases}
\end{align}

\pzc{Due to the second order time discretization, the solution obtained by the above iteration introduces an additional $O(\Dt^2)$ error from time marching.} 
Of course since $\Dt \sim \min \{\Delta x, \Delta y\}$ the EM-WaveHoltz solution is converging at the same rate as the spatial discretization but nevertheless it does have an additional error. This error is easily eliminated by a small modification which we discuss next.  \pzc{We only present the 2D TM model here but note that EM-WaveHoltz can be straightforwardly generalized to the full 3D model.}


\subsection{Eliminating the temporal error in EM-WaveHoltz}
For brevity we consider the energy conserving two dimensional TM model. 
\pzc{To eliminate the time-marching error, we first slightly modify the source term in the time-domain. We replace $\sin(\omega t^{n+\half})$ in \eqref{eq:yee_2d_tm_E} with }
\begin{align}
\label{eq:modified_sin}
S^\half = \frac{\omega\Dt}{2},\; S^{n+\half}=S^{n-\half} + \Dt\omega\cos(\omega t^n).
\end{align}
\pzc{Here, $S^{n+\half}$ is a second order approximation to $\sin(\omega t^{n+\half})$.} Using $S^{n+\half}$ instead of $\sin(\omega t^{n+\half})$ gives us a chance to eliminate the error due to the time discretization. 

Eliminating $H_x$ and $H_y$ in \eqref{eq:yee_2d_tm}, we have 
\begin{align}
\label{eq:yee_2d_tm_2nd}
&\frac{(\WE_z)_{i,j}^{n+1}-2(\WE_z)^n_{i,j}+(\WE_z)_{i,j}^{n-1} }{\Dt^2}+ L_h (\WE_z)^n_{i,j} \notag\\
=&- \frac{1}{\eps_{i,j}}(J_z)_{i,j}\frac{S^{n+\half}-S^{n-\half}}{\Dt}\notag\\
=& -\frac{1}{\eps_{i,j}}(J_z)_{i,j}\pzc{\omega}\cos(\omega t^n),
\end{align}
where 
\begin{align}
-L_h F_{i,j}&=\frac{1}{\eps_{i,j}\Dx}\Big(\frac{F_{i+1,j}-F_{i,j}}{\mu_{i+\half,j}\Dx}
		-\frac{F_{i,j}-F_{i-1,j}}{\mu_{i-\half,j}\Dx}\Big)\notag\\
		&+\frac{1}{\eps_{i,j}\Dy}\Big(\frac{F_{i,j+1}-F_{i,j}}{\mu_{i,j+\half}\Dy}
		-\frac{F_{i,j}-F_{i,j-1}}{\mu_{i,j-\half}\Dy}\Big).
\end{align}
\eqref{eq:yee_2d_tm_2nd} is an approximation to the second order form of the time-domain equation in EM-WaveHoltz. We now have the following theorem guaranteeing the convergence of the discrete iteration (for the energy conserving case)

\begin{thm}\label{thm:2d_tm}
Let $\nu^{\infty}$ be the solution to 
\begin{align}
\womega^2 \nu^\infty-L_h\nu^\infty = \omega\left(\frac{1}{\eps}J\right),
\end{align}
where 
\begin{align}
\pzc{\womega = \frac{\sin(\omega\Dt/2)}{\Dt/2}=\omega+O(\Dt^2).}
\end{align}
Further, let $\{-\lambda_j^2\}_{j=1}^N$ and $\{\psi_j\}_{j=1}^N$ be the eigenvalues and corresponding eigenfunctions of $L_h$, and $0<\lambda_1<\lambda_2<\dots<\lambda_N$.
Assume that $\omega$ is not a resonance and denote the relative distance to the closest resonance 
\begin{align}
\delta_h=\min_j|\lambda_j-\omega|/\omega>0.  
\end{align}
Then, for the energy conserving method \eqref{eq:yee_2d_tm}, with the filter \eqref{eq:filterYEE}, the Yee-EM-WaveHoltz iteration  $\nu^{(k+1)}=\Pi_h\nu^{(k)}$ with $\nu^{(0)}=0$ converges to $\nu^{\infty}$ as long as 
\begin{align}
\Dt\leq \frac{2}{\lambda_N+2\omega/\pi},\; \omega\Dt\leq \min(\delta_h,1).
\end{align}

Moreover, the convergence rate is at least $\rho_h=\max(1-0.3\delta_h^2,0.6)$.
\end{thm}
The proof of this Theorem is presented in Appendix \ref{sec:analysis_discrete}. 
We note that the first constraint on the timestep is essentially the standard CFL condition for an explicit method while the second condition could be very strict. 
\pzc{In fact, for all our numerical experiments, we only choose the $\Dt$ based on the CFL condition, and the violation of the second condition does not lead to problems. Hence, we conjecture that the second condition is not a practical limitation.}

Now, if we replace $\cos(\omega t^n)$  with $\cos(\bar{\omega}t^n)$ in \eqref{eq:modified_sin} with 
\begin{align}
\label{eq:omega_eliminate_time}
\bar{\omega} = \frac{2}{\Dt} \sin^{-1} \left( \frac{ \omega\Dt}{2} \right),
\end{align}
and modify the trapezoidal weights in the filter as
\begin{align}
\label{eq:modified_quadrature}
\widehat{\Pi}_h\nu = \frac{2\Dt}{T}\sum_{n=0}^M\left(\cos(\omega t^n)-\frac{1}{4}\right)\frac{\cos(\omega t^n)}{\cos(\bar{\omega} t^n)}\WE_z^n.
\end{align}
Then Theorem  \ref{thm:2d_tm} holds but the convergence is  to $\nu^{\infty}$ being the solution to the standard discretized frequency-domain problem
\begin{align}
\omega^2 \nu^\infty-L_h\nu^\infty = \omega\left(\frac{1}{\eps}J\right).
\end{align}
The derivation of this strategy is discussed in Appendix \ref{appendix:time_error} along with the proof of Theorem  \ref{thm:2d_tm}.
\pzc{An alternative strategy to eliminate the temporal error is suggested in \cite{garcia2021numerical}.}
\subsection{DG-EM-WaveHoltz}
The discontinuous Galerkin (DG) method, due to its high order accuracy, flexibility to use nonconforming meshes and its suitability for parallel implementation, has become increasingly popular for the simulation of time-domain wave propagation. As for the Yee scheme, DGTD can easily be turned into a frequency-domain solver using our approach. Here we use the time-domain DG method of \cite{HesthavenWarburton02,hesthaven2007nodal}. 

Consider Maxwell's equation in $d$-dimensions. Let $\Omega_j$ be an element, and $P^s(\Omega_j)$ be the space of polynomials at most degree $s$. Define $V_h^s(\Omega_j) = \big(P^s(\Omega_j)\big)^d$ to be the corresponding vector polynomial space. The DG method seeks the solution
$\WBE_h\in V_h^s(\Omega_j)$,  $\WBH_h\in V_h^s(\Omega_j)$ such that for any $\Bphi\in V_h^s(\Omega_j)$, $\Bpsi \in V_h^s(\Omega_j)$
\begin{subequations}
\label{eq:dg}
\begin{align}
&\int_{\Omega_j} \partial_t\WBE_h\cdot \Bphi d\BV=\int_{\Omega_j}\WBH_h\cdot \nabla\times(\frac{1}{\eps}\Bphi)d\BV\notag\\
&+\int_{\partial\Omega_j}(\widehat{\BH}\times \Bn)\cdot(\frac{1}{\eps}\Bphi) d\Bs
-\int_{\Omega_j}\sin(\omega t)\BJ\cdot\Bphi d\BV,
\end{align}
\begin{align}
&\int_{\Omega_j} \partial_t\WBH_h\cdot \Bpsi d\BV=-\int_{\Omega_j}\WBE_h\cdot \nabla\times(\frac{1}{\mu}\Bpsi)d\BV\notag\\
&-\int_{\partial\Omega_j}(\widehat{\BE}\times \Bn)\cdot(\frac{1}{\mu}\Bpsi )d\Bs.
\end{align}
\end{subequations}
Here, $\Bn$ is the outward pointing normal of a face and  $\widehat{\BH}$ and $\widehat{\BE}$ are numerical fluxes. A stable and accurate choice for the numerical fluxes is
\begin{align}
\widehat{\BH}=\{\WBH\}+\alpha[\WBE],\; \widehat{\BE}=\{\WBE\}+\beta[\WBH].
\end{align}
Here $\Bv^{\pm}$ denotes the two values on each side of a face, $\{\Bv\}=\frac{1}{2}(\Bv^{+}+\Bv^-)$ is the average and  $[\Bv]=\Bn^+\times \Bv^++\Bn^-\times \Bv^-$ is the jump. The semi-discretization \eqref{eq:dg} can be evolved in a method of lines fashion, using for example a Runge-Kutta or Taylor method as the time stepper. Depending on the time discretization it may be possible to eliminate the time error as discussed above but we don't pursue this here. Further, in the examples below we always use the trapezoidal rule to discretize the filter. 

\section{Numerical results}\label{sec:numerical}
In this section we demonstrate the performance of the EM-WaveHoltz methods on several examples in two and three dimensions. The $\sin$-forcing formulation is used in two dimensions and the $\cos$-forcing formulation is used in three dimensions, unless otherwise specified. \pzc{For all numerical examples, the time step size $\Dt$ is chosen based on the CFL conditions of the time-domain methods. Again we note that such timesteps violates the second condition in Theorem 1 but that this condition appears to be a technicality as none of the examples below are affected by this.} \revtwo{In this section we always use the Krylov accelerated version of the iteration. }{}

\subsection{\pzc{Comparison with the MEEP FDFD solver}}
\pzc{
We compare our Yee-EM-WaveHoltz code with the iterative FDFD solver of the open source C++ package MEEP \cite{oskooi2010meep}. Our code is implemented by combining EM-WaveHoltz with the FDTD code of the C library rbcpack \cite{rbc_pack_url}. Our code uses a self-implemented GMRES solver without restart. The FDFD solver of MEEP uses BICG-Stab($l$) method \cite{sleijpen1993bicgstab}. Both codes are executed in a serial-manner on a 2015 MacBook with 2.2 GHz Quad-Core Intel Core i7 cpu.}

 \begin{figure}[]
\centering
\includegraphics[width=0.24\textwidth,trim={2.2cm 0.0cm 2.6cm 0.8cm},clip]{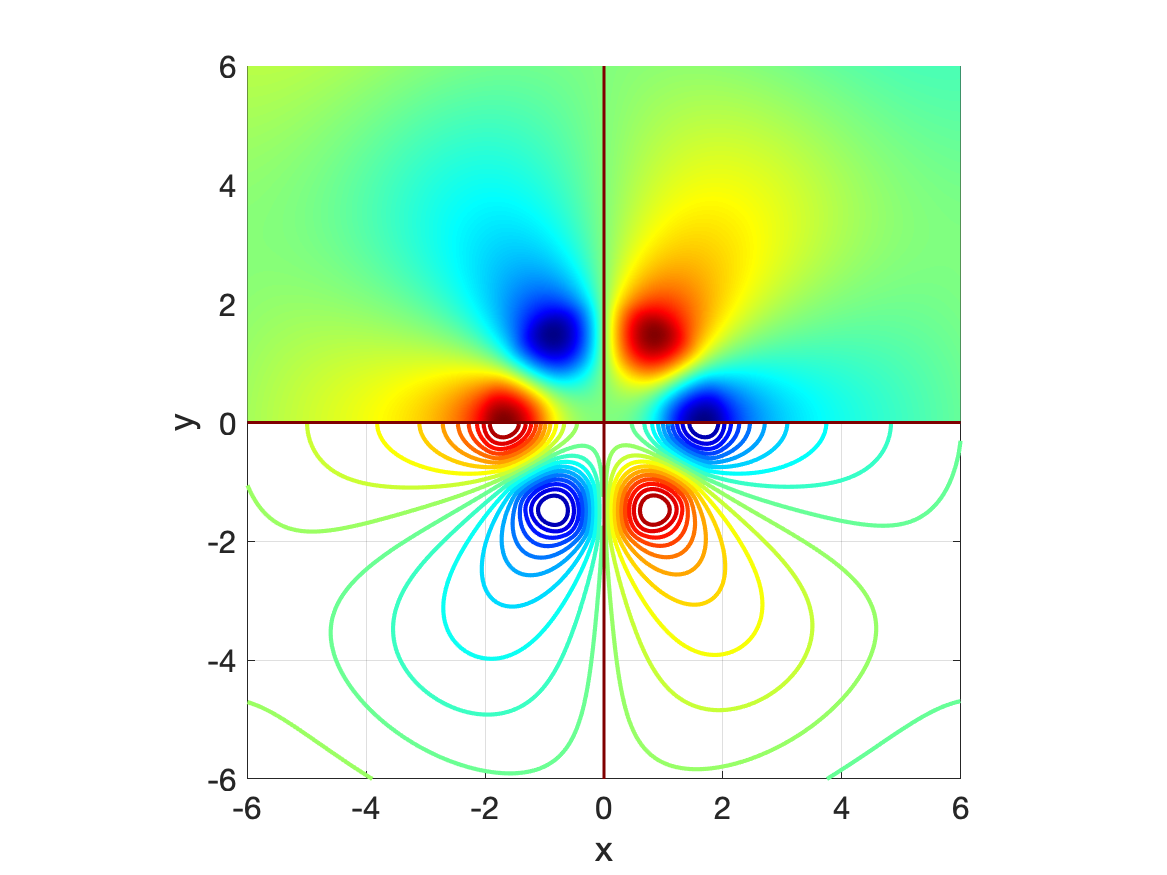}
\includegraphics[width=0.24\textwidth,trim={2.2cm 0.0cm 2.6cm 0.8cm},clip]{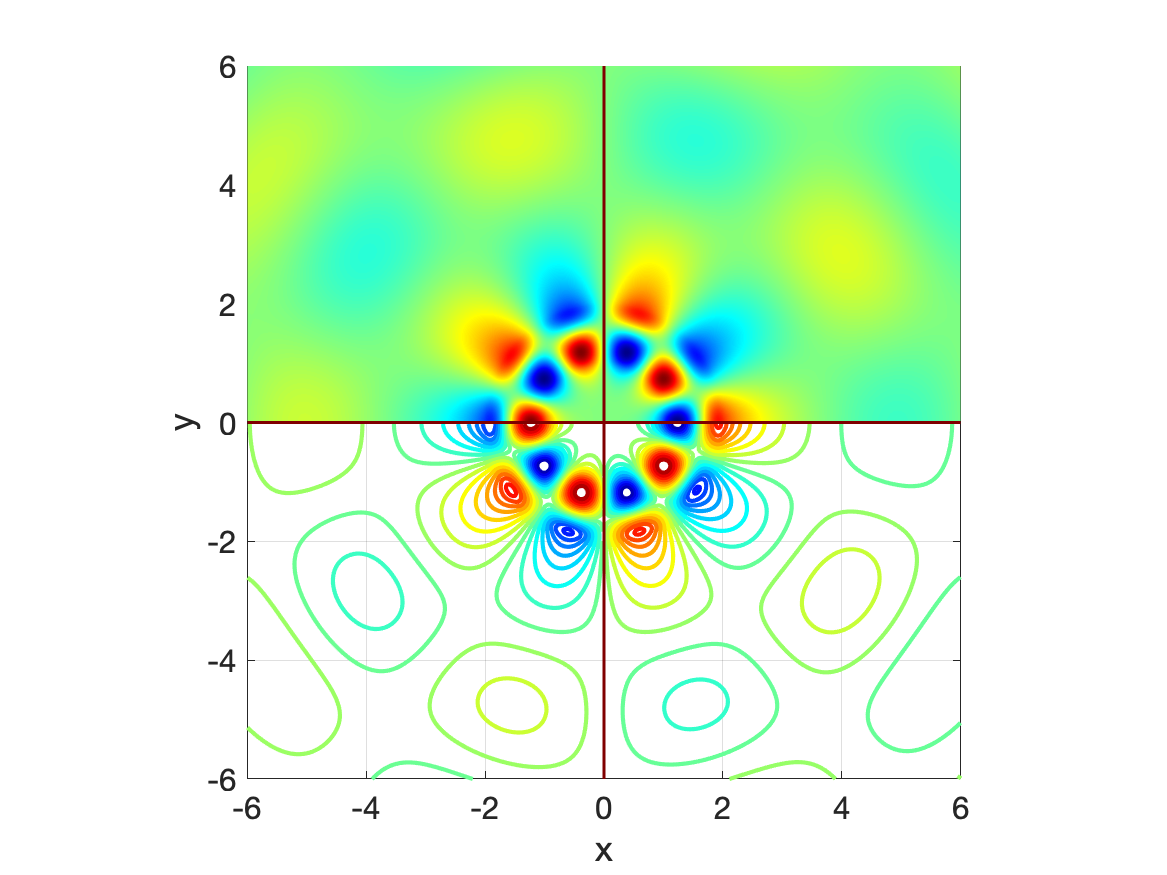}
\includegraphics[width=0.24\textwidth,trim={2.2cm 0.0cm 2.6cm 0.8cm},clip]{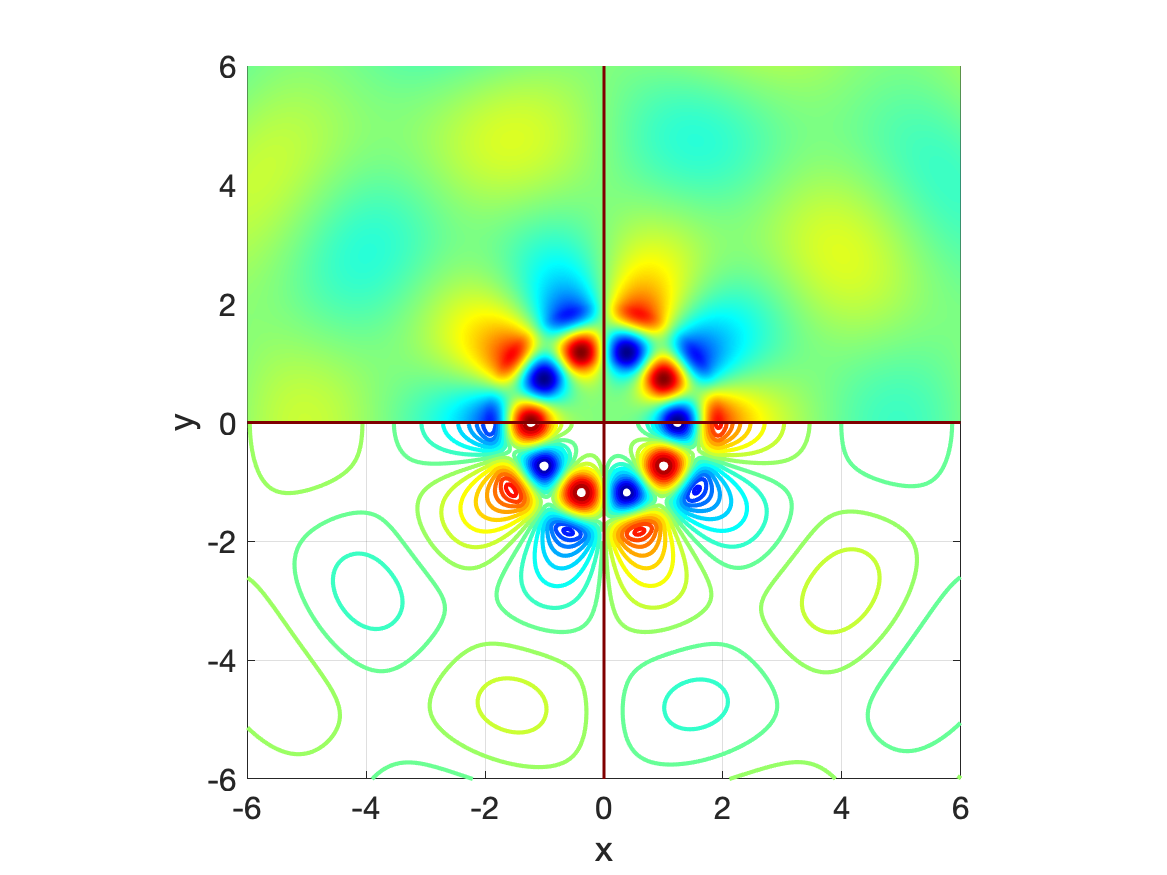}
\caption{The real part of the $E_z$ field (normalized). Top left figure, $\omega=\omega_0$, top right figure $\omega=2.24\omega_0$, bottom figure: $\omega=2.7\omega_0.$ Left: Yee-EM-WaveHoltz. Right: MEEP's FDFD solver. \label{fig:waveholtz_vs_meep_solution}. }
\end{figure}

\pzc{Following MEEP package's benchmark example for the FDFD code (see \cite{meep_url}), we consider a ring resonator and the 2D TM model. The computational domain is $[-6,6]^2$
 with nonreflecting boundary conditions. A ring resonator with $\eps_r=3.4^2$ is located at $\{(x,y):1\leq\sqrt{x^2+y^2}\leq2\}$. The permittivity outside the ring is $\eps=1$, and the permeability $\mu=1$ in the whole computational domain. Two point sources are placed at $(1.1,0)$ with magnitude $1$ and $(-1.1,0)$ with magnitude $-1$. Let $\omega_0=0.118\times 2\pi$. We consider $\omega=\omega_0,2.24\omega_0$ and $2.7\omega_0$. We use $N=120,240$ and $480$ grids in each direction.}
 
 \begin{table}[]
\caption{Computational time (sec)}
\label{tab:wh_vs_meep_time}
\centering
\begin{tabular}{|c|c|c|c|c|c|c|}
\hline
                         & $N$ & $\omega = \omega_0$ & $\omega = 2.24\omega_0$  & $\omega = 2.7\omega_0$\\
\hline 
\multirow{3}{*}{EM-WH}& $120$& $20$ & $12$ & $12$ \\
                             & $240$& $94$ & $55$ &$60$ \\
                             & $480$& $562$ & $326$ & $350$\\
                             \hline
\multirow{3}{*}{MEEP}&  $120$&  $11$ & $21$ & $34.14$\\
                     &  $240$&  $109$ & $180$ & $229$\\
                     &  $480$&  $1095$ & $1531$ & $2000$\\
\hline
\end{tabular}
\end{table}
\pzc{For both solvers, we set the relative tolerance as $10^{-7}$. For the Yee-EM-WaveHoltz, we use $\cos$-forcing and filter over $10$ periods. To obtain convergent results for all frequencies, we use $l=10$ in the BICG-Stab-($l$) FDFD solver. In the results displayed in Figure \ref{fig:waveholtz_vs_meep_solution}, we observe that the EM-WaveHoltz and the FDFD agree well. Table \ref{tab:wh_vs_meep_time} presents the computational time needed. The Yee-EM-WaveHoltz code is always faster except for $\omega_0$ and $N=120$. Its advantage increases with mesh refinement and the size of the frequency. In  Table \ref{tab:wh_vs_meep_iter}, we present the total number of iterations needed for convergence. The Yee-EM-WaveHoltz always needs \revtwo{fewer iterations}{less number of iterations} for convergence. Moreover, for a fixed frequency, the number of iteration needed by Yee-EM-WaveHoltz almost does not grow as the grid is refined, while the BICG-Stab-($10$) needs more iterations.}

\begin{table}[]
\caption{Total number of iterations with $10^{-7}$ relative residual}
\label{tab:wh_vs_meep_iter}
\centering
\begin{tabular}{|c|c|c|c|c|c|c|}
\hline
& $N$ & $\omega = \omega_0$ & $\omega = 2.24\omega_0$  & $\omega = 2.7\omega_0$\\
\hline 
\multirow{3}{*}{EM-WH}& $120$& $8$ & $11$ & $14$\\
                      & $240$& $8$ & $11$ & $15$\\
                      & $480$& $8$ & $11$ & $15$\\
                             \hline
\multirow{3}{*}{BICG-Stab($l$)}&  $120$& $144$& $256$ & $405$\\
                     &  $240$& $280$& $449$ & $557$\\
                     &  $480$& $586$& $800$ & $1092$\\
\hline
\end{tabular}
\end{table}
 
\subsection{\pzc{Comparison with a direct  FDFD solver}}
\begin{figure}[]
\centering
\includegraphics[width=0.47\textwidth]{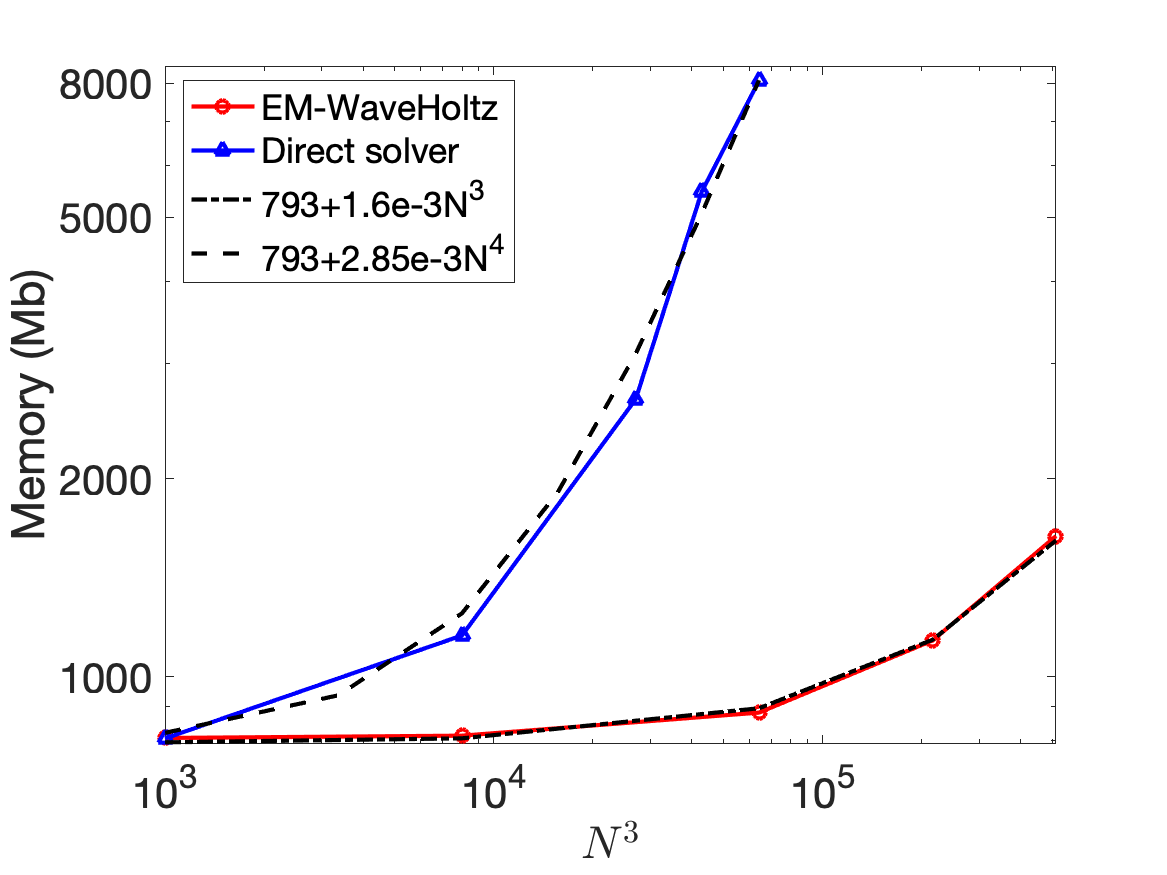}
\caption{Peak memory for EM-WaveHoltz and a direct solver based on sparse multifrontal LU factorization for a 3D problem with PEC boundary conditions}
\label{fig:memory_compare}
\end{figure}
\pzc{Sparse multifrontal direct FDFD solvers are fast, however, for the full 3D problem and with increasing frequency such solvers quickly become too large to fit in memory. To demonstrate this we consider a 3D problem with $\omega=12.5$, PEC boundary conditions and a source $J_x=\omega y(y-1)z(z-1)+\frac{2}{\omega}( y(y-1)+z(z-1) )$, $J_y=\omega x(x-1)z(z-1)+\frac{2}{\omega}( x(x-1)+z(z-1) )$ and $J_z=\omega x(x-1)y(y-1)+\frac{2}{\omega}( x(x-1)x+y(y-1) )$. We use $N$ grid points in each direction. We implement an Yee-EM-WaveHoltz code and a direct FDFD code in Julia.  Both codes share exactly the same spatial discretization subroutines based on sparse matrices. \revtwo{GMRES solver with relative tolerance $10^{-8}$ is applied in the EM-WaveHoltz code.}{} The direct solver uses Julia's sparse multifrontal LU factorization, which calls SuiteSparse \cite{davis2011university}. As shown in Figure \ref{fig:memory_compare}, we observe that the peak memory needed by the EM-WaveHoltz scales  roughly  as  $O(DOF)=O(N^3)$, and the peak memory needed by the direct solver scales roughly as $O(DOF^{4/3})=O(N^4)$.}

\subsection{\pzc{Grid convergence} of Yee-EM-WaveHoltz}
We \pzc{consider the 2D TM model} and non-dimensionalize the equations so that $\eps=\mu=1$ and manufacture a forcing 
\begin{align*}
\pzc{
J}&\pzc{=16\omega x^2(x-1)^2y^2(y-1)^2+\frac{32}{\omega}\big(}
 \\ &\pzc{(6x^2-6x+1)y^2(y-1)^2
+(6y^2-6y+1)x^2(x-1)^2\big)}
\end{align*} so that the exact solution is $E_z(x,y)=16x^2(x-1)^2y^2(y-1)^2$. This solution is compatible with perfect electric conductor (PEC) boundary conditions on the domain $[0,1]^2$. We apply the Yee scheme and the EM-WaveHoltz iteration with GMRES acceleration. \revtwo{The relative tolerance of the GMRES solver is set as $10^{-10}$.}{}

To test the convergence for a) one frequency in one solve, and b) multiple frequencies in one solve, we perform a grid refinement study at fixed frequencies $\omega_1 = 5.5$, $3\omega_1$ and $7\omega_1$.  
In Figure \ref{fig:grid_convergence},  \pzc{without eliminating the temporal error}, we display how the error is decreased as the grid size is reduced. As expected, for both of one frequency in one solve and multiple frequencies in one solve, we observe second order convergence. \pzc{For the same frequency and the same mesh, the magnitude of the errors for one frequency and multiple frequencies in one solve are close to each other.}

\pzc{ We also use this example to investigate the influence of the temporal error in the EM-WaveHoltz method. In Table \ref{tab:compare_error_with_fdfd}, due to the temporal error, we observe that the error of the direct FDFD solver is smaller than the Yee-EM-WaveHoltz for for $\omega=16.5$ and $\omega=37.5$, but slightly bigger for $\omega=5.5$. For $\omega=5.5$, it is likely that the sign for the temporal error and the spatial error are different.}

\pzc{Finally, we verify the effectiveness of our strategy to eliminating the temporal error. As shown in Table \ref{tab:eliminate_time_error}, the difference between the FDFD solution and the Yee-EM-WaveHoltz solution after eliminating temporal error are at most $O(10^{-12})$, which is much smaller than the numerical error.}

\begin{figure}[]
\centering
\includegraphics[width=0.47\textwidth]{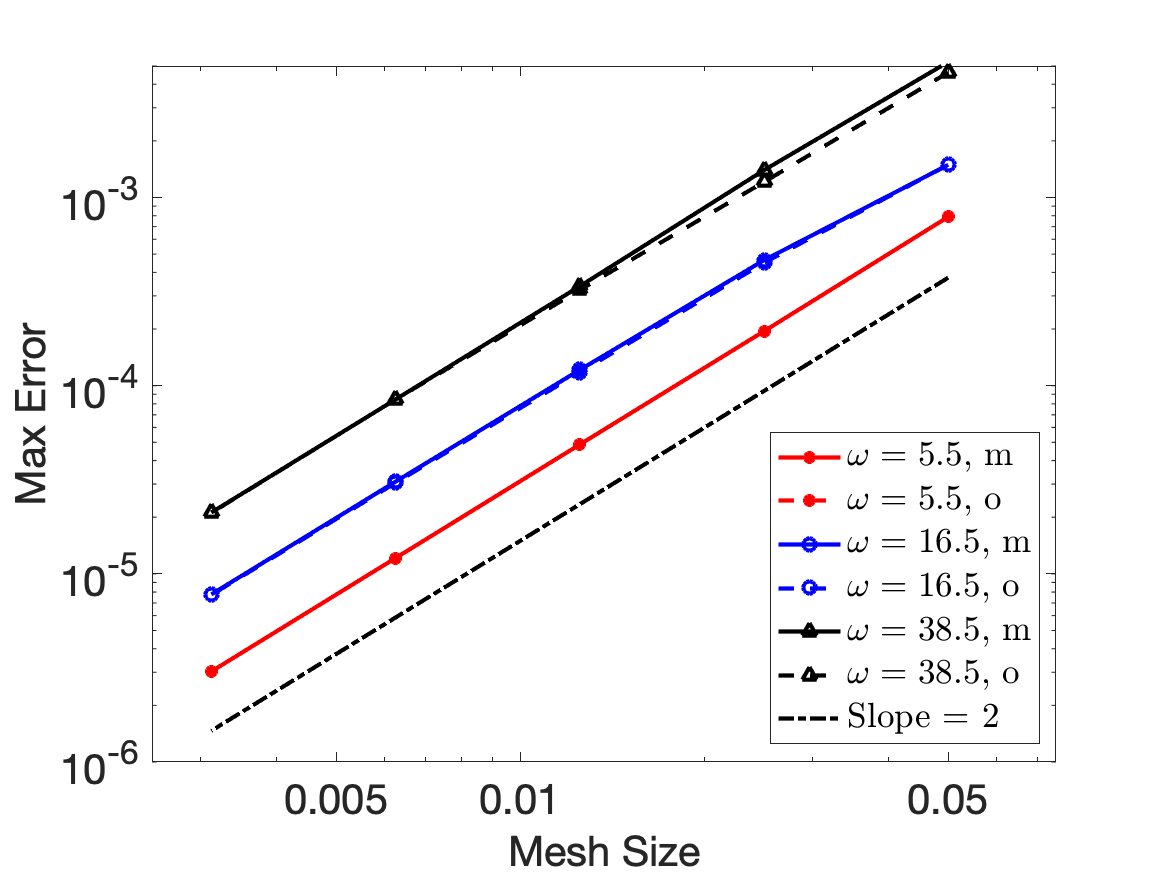}
\caption{Grid convergence of a manufactured solution. Here ``o" stands for one frequency in one solve, and ``m" stands for multiple frequencies in one solve. The errors displayed are for $E_z$ for a 2D TM model. }
\label{fig:grid_convergence}
\end{figure}

\begin{table}[!htb]
\caption{Errors for the manufacture solution of 2D TM model with $N=160$ grids in each direction.}
\label{tab:compare_error_with_fdfd}
\centering
\begin{tabular}{|c|c|c|c|c|c|c|}
\hline
$\omega$ & 5.5 &16.5  & 35.5  \\
\hline 
EM-WH with temporal error & 1.21e-5 & 3.09e-5 & 8.45e-5 \\
\hline
Direct-FDFD & 1.73e-5& 2.24e-6 & 2.26e-7 \\
\hline
\end{tabular}
\vspace*{0.5\baselineskip}
\newline
\caption{Difference between Yee-EM-WaveHoltz solution after eliminating the temporal error and the FDFD solution with $N=160$ grids in each direction}
\label{tab:eliminate_time_error}
\centering
\begin{tabular}{|c|c|c|c|c|c|c|}
\hline
$\omega$ & 5.5 &16.5  & 37.5  \\
\hline 
Difference & 7.38e-15 & 9.22e-14 & 2.24e-12 \\
\hline
\end{tabular}
\end{table}

\subsection{Plane wave scattering and $p$-convergence of DG-EM-WaveHoltz\label{sec:DG}}
Next we combine the EM-WaveHoltz iteration with the upwind nodal discontinuous Galerkin method \cite{HesthavenWarburton02}. \pzc{We consider the 2D TM-model and the scattering wave from a PEC disk due to the incident plane wave $\revtwomath{E_z^{\textrm{inc}}=e^{-i\omega x}}$ with $\omega=15$. The radius of the disk is $a=0.25$. The exact solution of this problem is a Mie series (see e.g. \cite{petropoulos2000reflectionless}), and is presented in Fig. \ref{fig:scatter_dg}). The incident wave is imposed by setting the boundary value $E_z^{bc}$ equal to the exact solution. In the EM-WaveHoltz formulation, the boundary condition of the time-domain problem is defined as $\Re\{E_z^{\textrm{bc}}\}\cos(\omega t)+\Im\{E_z^{\textrm{bc}}\}\sin(\omega t)$. With this choice, one can show that the EM-WaveHoltz converges to the real part of the frequency-domain solution. For a  $p$-th degree polynomial spatial discretization  we use a $p+1$-th order Taylor series method in time and filter over $5$ periods.}

\pzc{We perform a $p$-convergence study with the unstructred mesh in Fig. \ref{fig:scatter_dg}. The maximum error and the number of iterations for convergence with relative tolerance $10^{-8}$ are shown in Table \ref{tab:dg_p_convergence}. As the polynomial order $p$ increases, the error decays, and high order schemes achieves $O(10^{-6})$ error on this relatively coarse mesh. As the polynomial order $p$ increases, number of points per wavelength grows, but the total number of iterations for convergence does not grow.}

\begin{table}[!thb]
\caption{Maximum error and the total number of iterations for the scattering wave from a PEC disk, with DG-EM-WaveHoltz and $p$-th order polynomial}
\label{tab:dg_p_convergence}
\centering
\begin{tabular}{|c|c|c|c|c|c|c|c}
\hline
$p$  &  3 & 4 & 5 & 6 & 7 & 8  \\ \hline
Error & 2.86e-2 & 4.65e-3 & 6.89e-4 & 6.40e-5 & 7.89e-6 & 1.37e-6 \\ \hline
Iterations & 45 & 39 & 38 & 38 & 37 &37\\
\hline
\end{tabular}
\end{table}

\begin{figure}[!htb]
\centering
\includegraphics[width=0.24\textwidth,trim={2.2cm 0.0cm 2.6cm 0.9cm},clip]{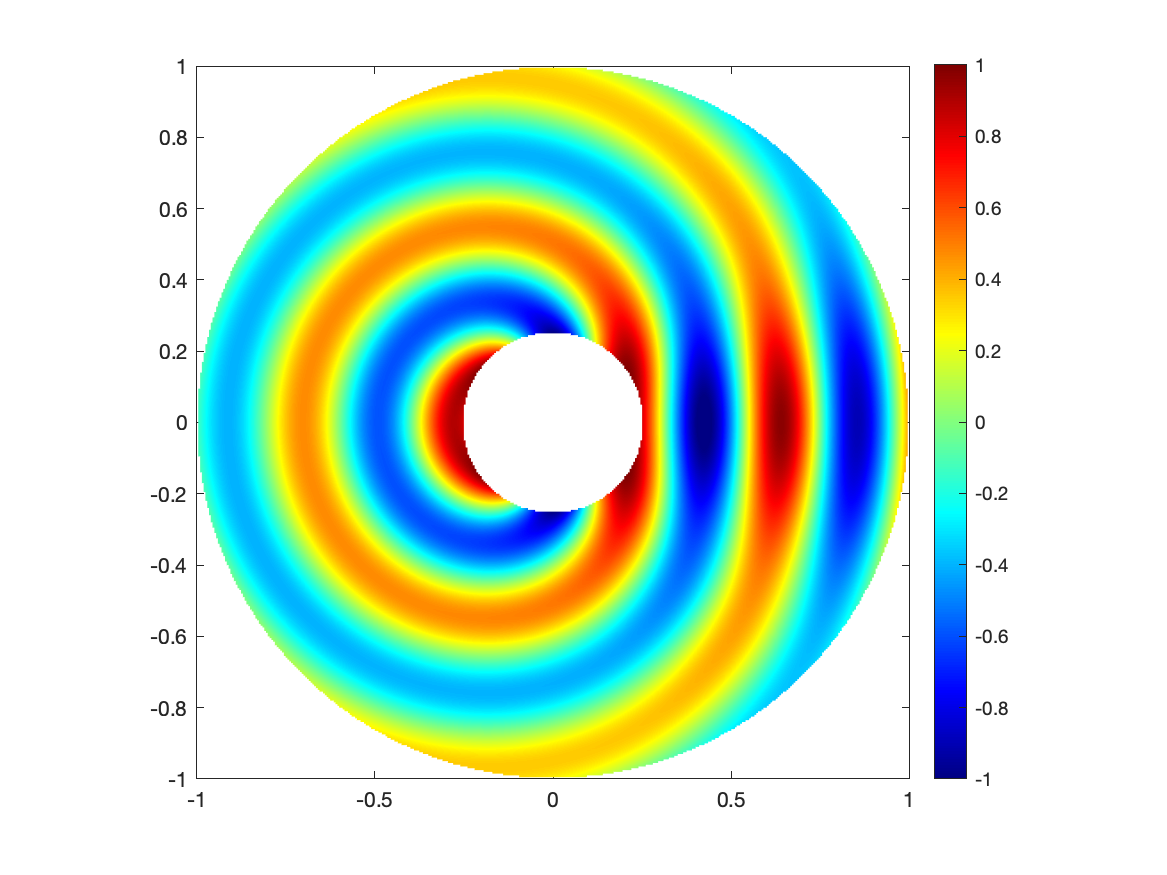}
\includegraphics[width=0.24\textwidth,trim={2.2cm 0.0cm 2.6cm 0.9cm},clip]{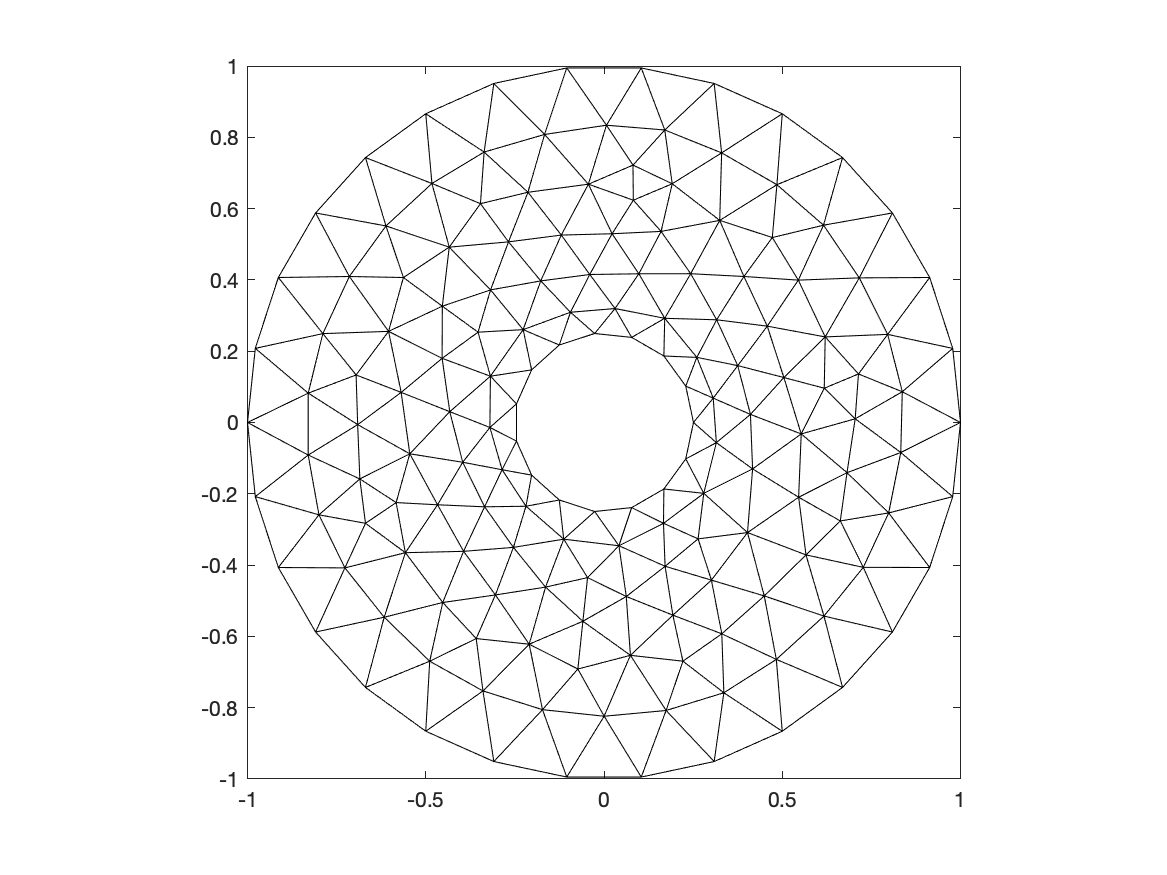}
\caption{Left figure: exact solution for $\omega=15.0$ of the scattering problem. Right figure: Unstructured mesh.\label{fig:scatter_dg}}
\end{figure}

\subsection{\revtwo{Plane wave scattering with incident fields and Yee scheme}{}}
\revtwo{We consider the same plane wave scattering problem as Section \ref{sec:DG} with the incident plane wave $E_z^{\textrm{inc}}=e^{-i\omega x}$ and $\omega=15$. To show the capability of using a field source, we apply the Yee-EM-WaveHoltz method and impose the incident fields following the total field/scattered field formulation in \cite{Taflove:2005fk}. We split the computational domain into a total field region $[-0.5,0.5]^2$ and a scattered field region. The incident wave is imposed through the interface condition between the two regions, and the double absorbing boundary layer (DAB) by Hagstrom et al. \cite{lagrone2016double} is applied to impose the nonreflecting boundary conditions. The setup of the Yee scheme is illustrated in the left figure of Fig. \ref{fig:scatter_yee}. 
To impose the right going incident wave, we follow similar arguments to \eqref{eq:maxwell_time_details} and define the corresponding time-domain incident fields as}{}
\begin{align}
&\revtwomath{\WBE_z^{\textrm{inc}}= \cos(\omega x)\sin(\omega t)-\sin(\omega x)\cos(\omega t)}\notag\\
&\qquad\revtwomath{=\sin(\omega(t-x))=\Im\{e^{i\omega(t-x)}\},}\notag\\
&\revtwomath{\WBH^{\textrm{inc}}_x = 0,\quad \WBH^{\textrm{inc}}_y = -\WBE_z^{\textrm{inc}}.}
\end{align}
\revtwo{We use a $401\times 401$ uniform mesh, and set the relative tolerance of the GMRES solver as $10^{-8}$. The numerical solution matches the exact solution very well (see right figure of Fig. \ref{fig:scatter_yee}).}{}
\newsavebox{\tempfig}
\begin{figure}[!htb]
\centering
 \savebox{\tempfig}{\includegraphics[width=0.23\textwidth,trim={2.2cm 0.0cm 2.6cm 0.9cm},clip]{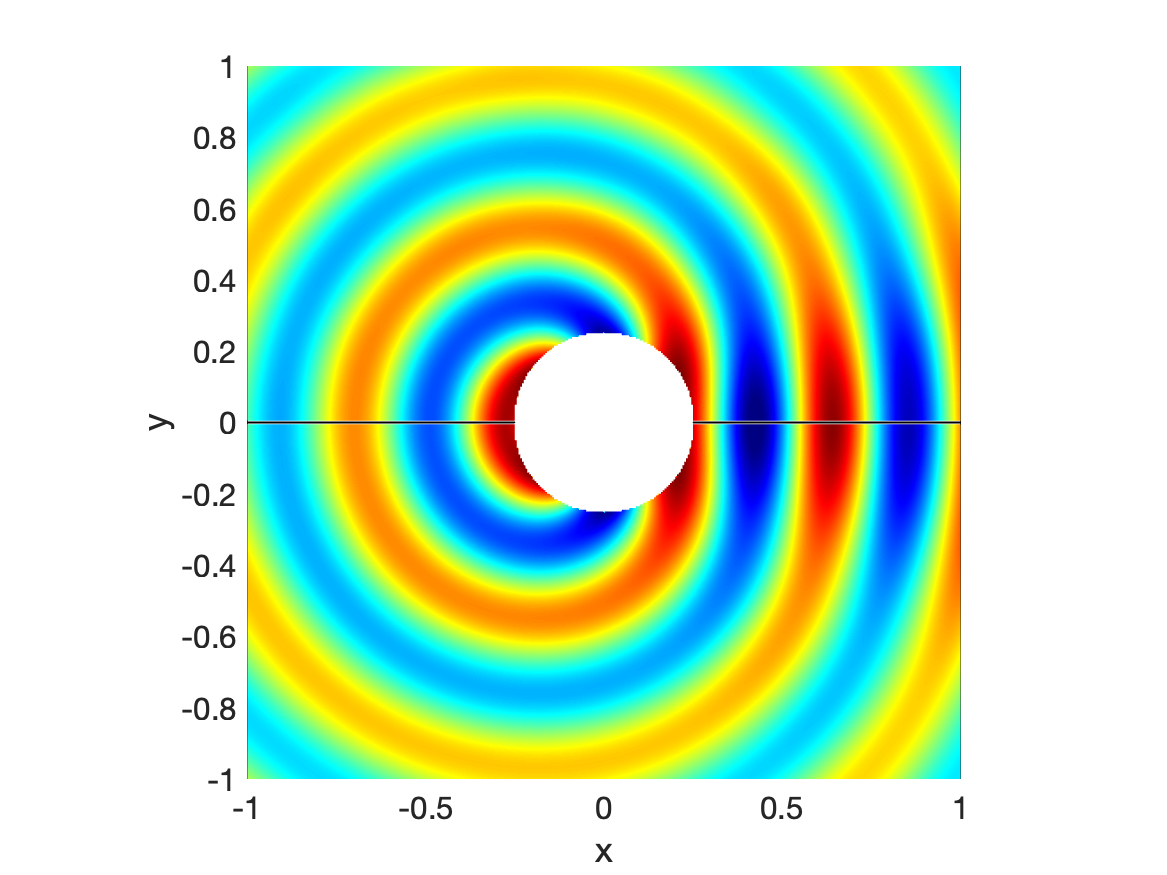}}
\raisebox{\dimexpr\ht\tempfig-\height}{\includegraphics[width=0.25\textwidth,trim={0.0cm 0.0cm 0.2cm 0.0cm},clip]{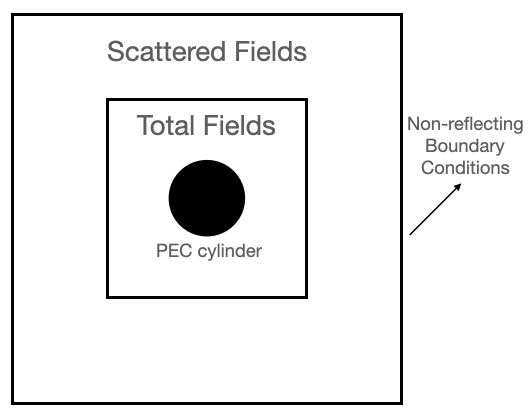}}
\includegraphics[width=0.23\textwidth,trim={2.2cm 0.0cm 2.6cm 0.9cm},clip]{picture/mie_yee.png}
\caption{\revtwomath{Left figure:  the set-up of the total field/scattered field formulation. Right figure: the real part of the scattered wave with $\omega=15$, numerical solution on the top and exact solution on the bottom.\label{fig:scatter_yee}}}
\end{figure}

\subsection{Number of iterations for different frequencies and boundary conditions in two dimensions}\label{sec:2d_convergence}
In this experiment we solve the 2D TM model with the source 
\begin{align}
J_z = -\omega \exp\left(-\sigma( (x-0.01)^2+(y-0.015)^2)\right),
\end{align}
where $\sigma = \max(36,\omega^2)$, $\eps=\mu=1$, and  the computational domain is $[-1,1]^2$. Here we sweep over the frequencies \mbox{$\omega=k+\half$}, \mbox{{$1\leq k\leq 100$}}. We use the GMRES accelerated Yee-EM-WaveHoltz iteration and to keep the solution reasonably well resolved we use $8\lceil\omega\rceil$ grids in each directions, \pzc{where $\lceil \omega \rceil$ is the smallest integer larger than $\omega$}.

We solve this problem with $6$ different boundary conditions: (1) 4 open boundaries, (2) 1 PEC boundary and 3 open boundaries, (3) 2 parallel PEC boundaries and 2 open boundaries, (4) 2 PEC boundaries next to each other forming a PEC corner and 2 open boundaries, (5) $3$ PEC boundaries and $1$ open boundary, (6) 4 PEC boundaries. The rationale here is that in problems (1), (2) and (4) there will not be any opposing PEC walls where waves can be ``trapped" while in the other three problems there are.  

\begin{figure}[!htb]
\centering
\includegraphics[width=0.47\textwidth]{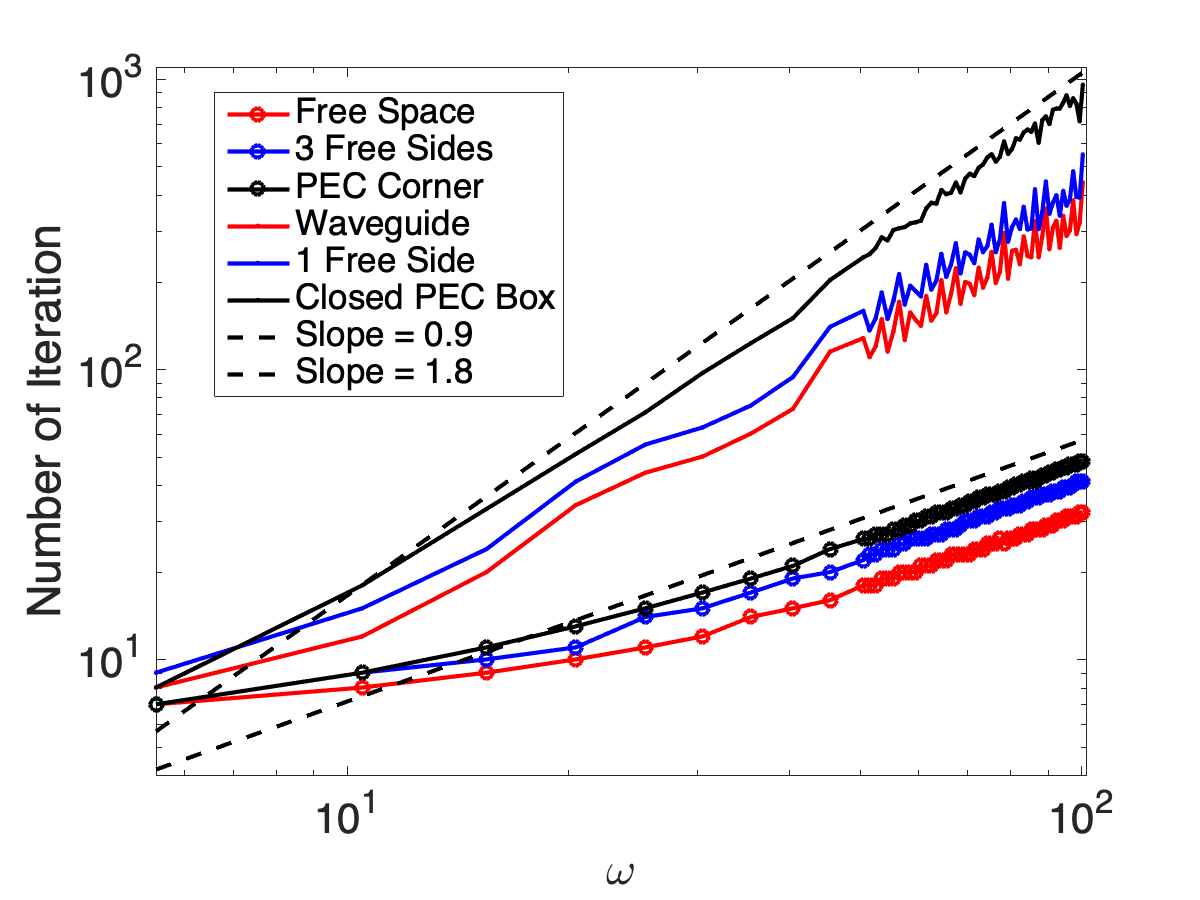}
\caption{Number of iterations as a function of frequency for the six different 2D TM problems. \label{fig:iteration_vs_omega_2d}}
\end{figure}

We also note that in the open directions we employ the optimally accurate double absorbing boundary layer (DAB) by Hagstrom et al. \cite{lagrone2016double}. The order of approximation in the DAB layers we use is 10 which virtually makes the non-reflecting boundary conditions exact.

In the EM-WaveHoltz iteration, we use $10$ periods so that $T=10\frac{2\pi}{\omega}$. This reduces the memory consumption in GMRES by a factor 10 and reduces the number of iterations by nearly a factor of 10 (the cost per iteration of course goes up by 10 as well). In Fig. \ref{fig:iteration_vs_omega_2d}, the number of iteration required to reduce the relative residual below $10^{-7}$ are presented. We observe that for the problems without trapped waves, the number of iterations scales approximately as $\omega^{0.9}$.  For the problems with trapped waves, the iteration converges slower and the number of iterations scales as approximately $\omega^{1.8}$.

\subsection{Number of iterations for different frequencies and boundary conditions in three dimensions}\label{sec:3d_convergence}
In this example we solve the 3D Maxwell's equation with a source 
\begin{align}
J_x = -\omega \exp\left(-\sigma(x^2+y^2+z^2)\right),\; J_y=J_z=0.
\end{align}
Here  $\sigma = \max(36,\omega^2)$, $\eps=\mu=1$ and the  computational domain is  $[-1,1]^3$. 

To measure how the number of iterations grow with the frequency, three different problems are considered: (1) an open domain, (2) two parallel PEC plates, (3) five PEC boundaries and one free side on the most left side. Again, we still apply the highly accurate double absorbing boundary layer (DAB) for non-reflecting boundary conditions. 
The order of the DAB layers is set as $5$ guaranteeing that the error of the non-reflecting boundary conditions is well below the discretization error.

We sweep over frequencies and use the GMRES accelerated Yee-EM-WaveHoltz method with the $\cos$-forcing. To have a well resolved solution we use $4\lceil\omega\rceil$ elements in each direction, \pzc{where $\lceil\omega\rceil$ is the smallest integer larger than $\omega$}. In the EM-WaveHoltz iteration, we set $T$ as $5$ periods.

In Figure \ref{fig:iteration_vs_omega_3d}, iteration numbers to reduce the relative residual below $10^{-6}$ are presented. The total number of iterations is estimated to scale as $\omega^{0.9}$ for the open problem, $\omega^{1.9}$ for the parallel PEC plate problem and $\omega^{2.5}$ for the problem with one  free side.

We also use the the $\sin$-forcing to simulate the open domain problem with the same mesh and the error tolerance. We observe that the number of iteration is exactly the same as the $\cos$-forcing, though the relative residual is slightly different for high frequencies.
\begin{figure}[!htb]
\centering
\includegraphics[width=0.5\textwidth]{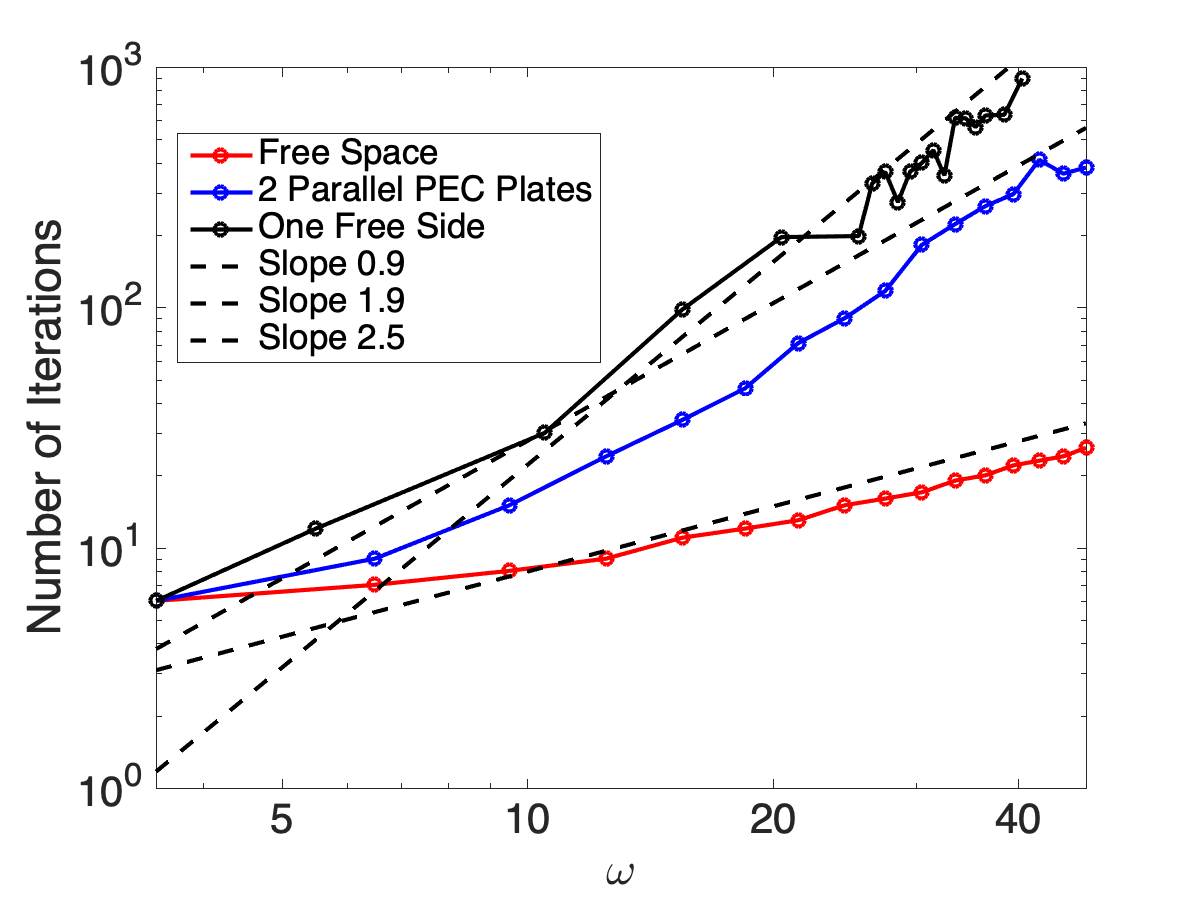}
\caption{Number of iterations as a function of frequency for different 3D problems.}
\label{fig:iteration_vs_omega_3d}
\end{figure}

\subsection{Number of iterations for different points per wavelength\label{sec:iteration_vs_ppw}}
\begin{table}[!thb]
\caption{Number of iterations for different number of points per wavelength, 2D-TM model, \pzc{$\lceil\omega\rceil=\textrm{ceil}(\omega)$.}}
\label{tab:iteration_vs_ppw_2d}
\centering
\begin{tabular}{|c|c|c|c|c|c|c|c}
\hline
Boundaries & $N$  &$2\lceil\omega\rceil$  &$4\lceil\omega\rceil$& $6\lceil\omega\rceil$& $8\lceil\omega\rceil$ \\
\hline
open & $\omega=12.5$  & 9 & 9 & 9 & 9\\ \hline
PEC &$\omega=12.5$  & 11 & 11 & 11 &11\\ \hline
open & $\omega=25.5$ & 13 & 12 &12 & 12\\ \hline
PEC &$\omega=25.5$& 25 & 24 &24 & 24\\ \hline
\end{tabular}
\end{table}
Here, we fix the frequency and \pzc{more systematically} investigate the number of iterations needed for convergence for different number of grid points per wavelength. We use the Yee-EM-WaveHoltz method with GMRES acceleration, and $\eps=\mu=1$ is considered. For the 2D-TM model, we consider the computational domain $[-1,1]^2$ and the source
\begin{align}
J_z = \omega\exp(-144(x^2+y^2)),
\end{align}
with either 4 open boundaries or 4 PEC boundaries. For the 3D model, we consider the computational domain $[-1,1]^3$ and the source
\begin{align}
J_x = -\omega \exp\left(-144(x^2+y^2+z^2)\right),\; J_y=J_z=0,
\end{align}
with either 6 open boundaries or 6 PEC boundaries. In each direction, we use $\pzc{N+1}$ grid points. We take $T=10\frac{2\pi}{\omega}$.
The stopping criteria is that the relative residual falls below $10^{-8}$ for the 2D TM-model and $10^{-5}$ for the 3D model. 

The results are presented in Table \ref{tab:iteration_vs_ppw_2d} and Table \ref{tab:iteration_vs_ppw_3d}. When considering an open problem at a fixed frequency, we observe that the number of iterations does not change as the number of grid points per wavelength is increased. For the PEC problem in three dimensions the number of \revtwo{}{of} iterations are reduced sightly as the resolution is increased and for the 2D PEC problem it does not change significantly. Based on this experiment and \pzc{other experiments},  our observation is that the algorithm is robust with respect to resolution (but of course the discretization error will depend on the resolution.) 

\begin{table}[!htb]
\caption{Number of iterations for different number of points per wavelength, 3D. \label{tab:iteration_vs_ppw_3d}}
\centering
\begin{tabular}{|c|c|c|c|c|c|c|c}
\hline
Boundaries & $N$  &$2\lceil\omega\rceil$  &$4\lceil\omega\rceil$& $6\lceil\omega\rceil$& $8\lceil\omega\rceil$ \\
\hline
open & $\omega=12.5$  & 5 & 5 & 5 & 6\\ \hline
PEC &$\omega=12.5$  & 26 & 24 & 23 &22\\ \hline
open & $\omega=25.5$ & 7 & 7 &7 & 7\\ \hline
PEC &$\omega=25.5$& 174 & 153 &150 & 135\\ \hline
\end{tabular}
\end{table}

\subsection{Smaller Krylov subspaces by longer filter time}
As we mentioned before, we can filter over multiple periods $T=N\frac{2\pi}{\omega}$, which allows the further propagation of the wave. We consider $T=N\frac{2\pi}{\omega}$ with $N=1,3,5$ for 2D and 3D open domain problem. The setup of this test is the same as Section \ref{sec:2d_convergence} for 2D and Section \ref{sec:3d_convergence} for 3D. We scan over different frequencies and apply the GMRES accelerated Yee-EM-WaveHoltz. The total number of iteration allowed  is set as $200$ in 2D and $100$ in 3D. As can be seen in Figure \ref{fig:memory_lean} for high frequencies both the 2D and the 3D solver, when using $T=\frac{2\pi}{\omega}$, fails to converge to the desired tolerance before reaching the maximum number of iterations. 

In Figure \ref{fig:memory_lean}, we present unscaled number of iterations against the frequency and observe that the number of iterations decays as the  propagation time $T$ in the time-domain grows. To further quantify the relation between the computational cost and $T=N\frac{2\pi}{\omega}$, we scale the number of iterations by $N$ and present the result in Figure \ref{fig:memory_lean_scaled}. We observe that for $N = 3$ and 5 the scaled curves collapse, implying that the the total computational time is approximately the same. Thus, without increasing the computational cost, filtering over longer time can decrease the number of iterations, which  in turn reduces the size of the Krylov subspace used by GMRES. 

\begin{figure}[htb]
\centering
\includegraphics[width=0.24\textwidth,clip]{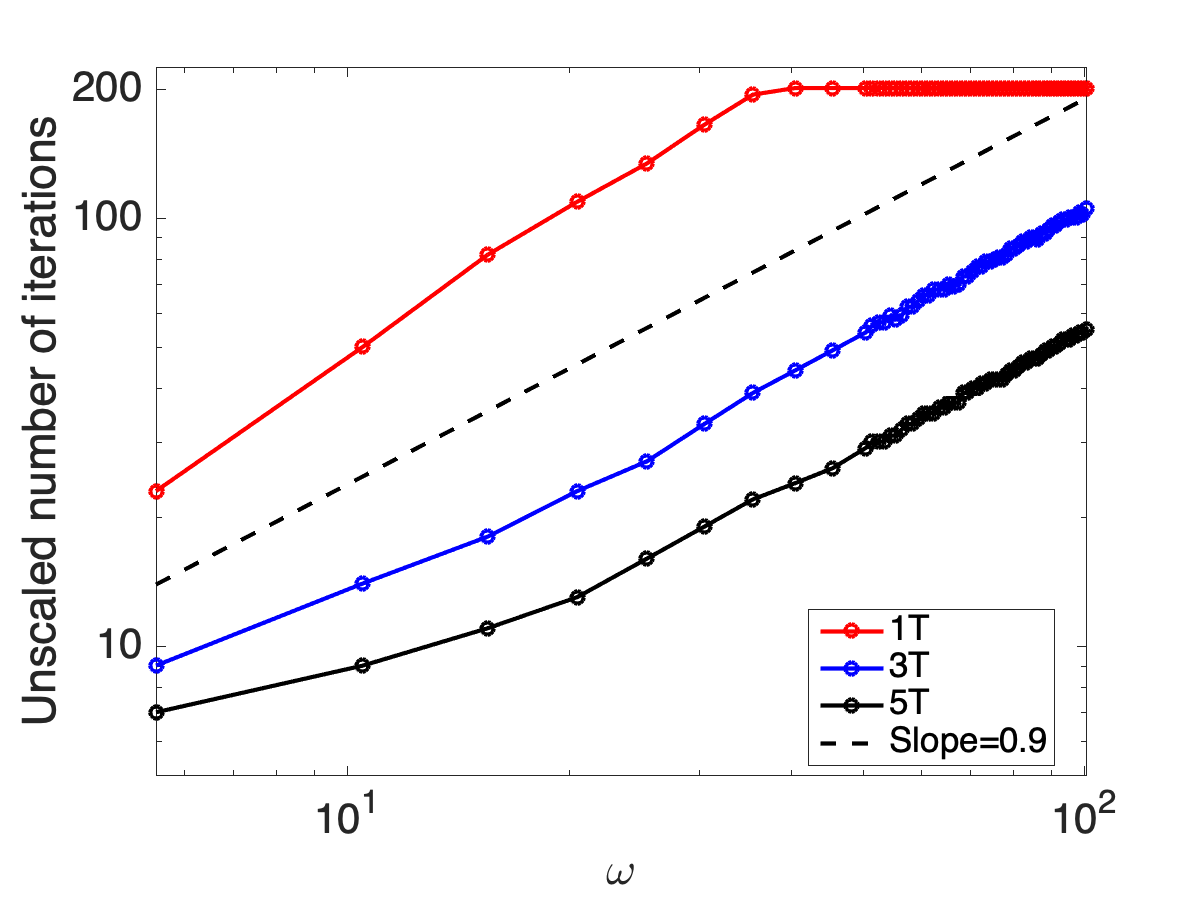}
\includegraphics[width=0.24\textwidth,clip]{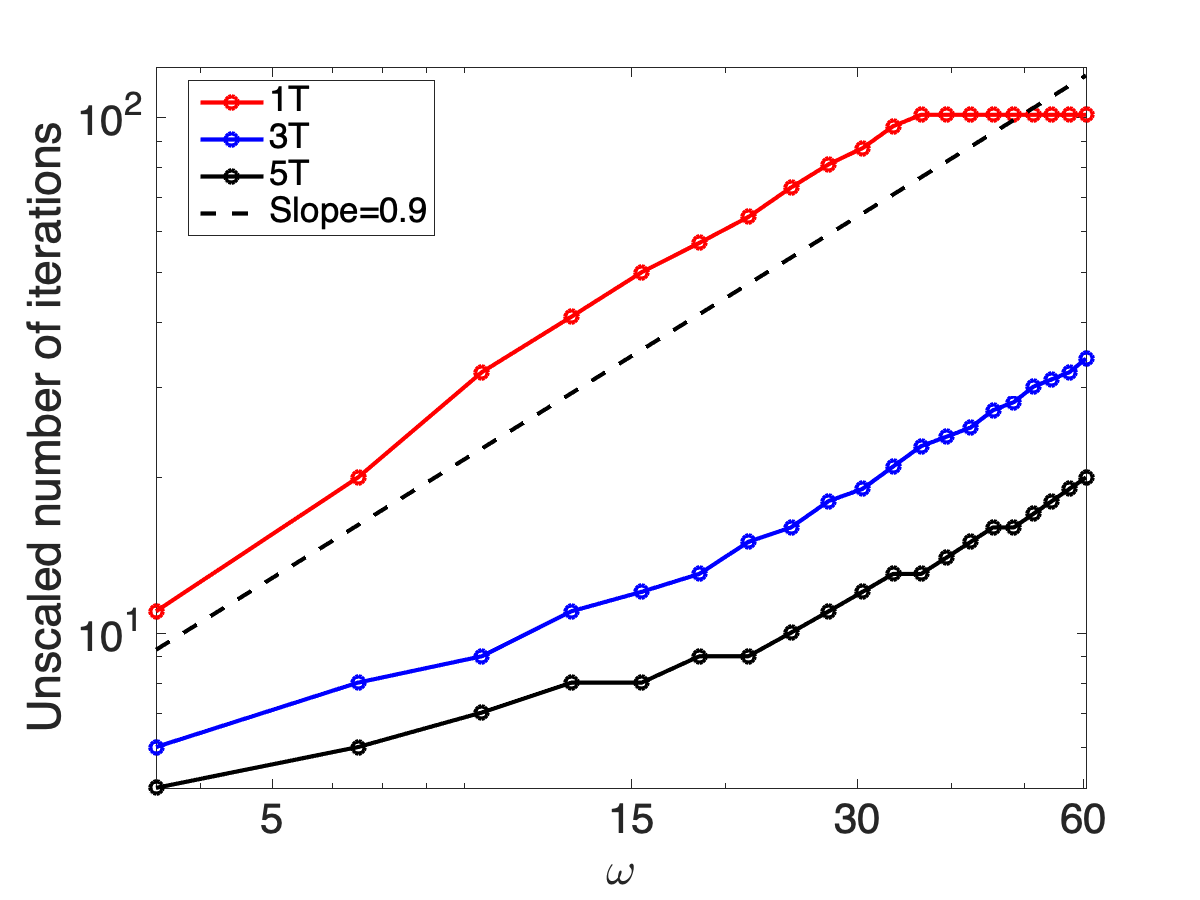}
\caption{Unscaled number of iterations as a function of frequency for different filtering time. Top: 2D open problem. Bottom: 3D open problem.  \label{fig:memory_lean}}
\end{figure}
\begin{figure}[htb]
\centering
\includegraphics[width=0.24\textwidth,clip]{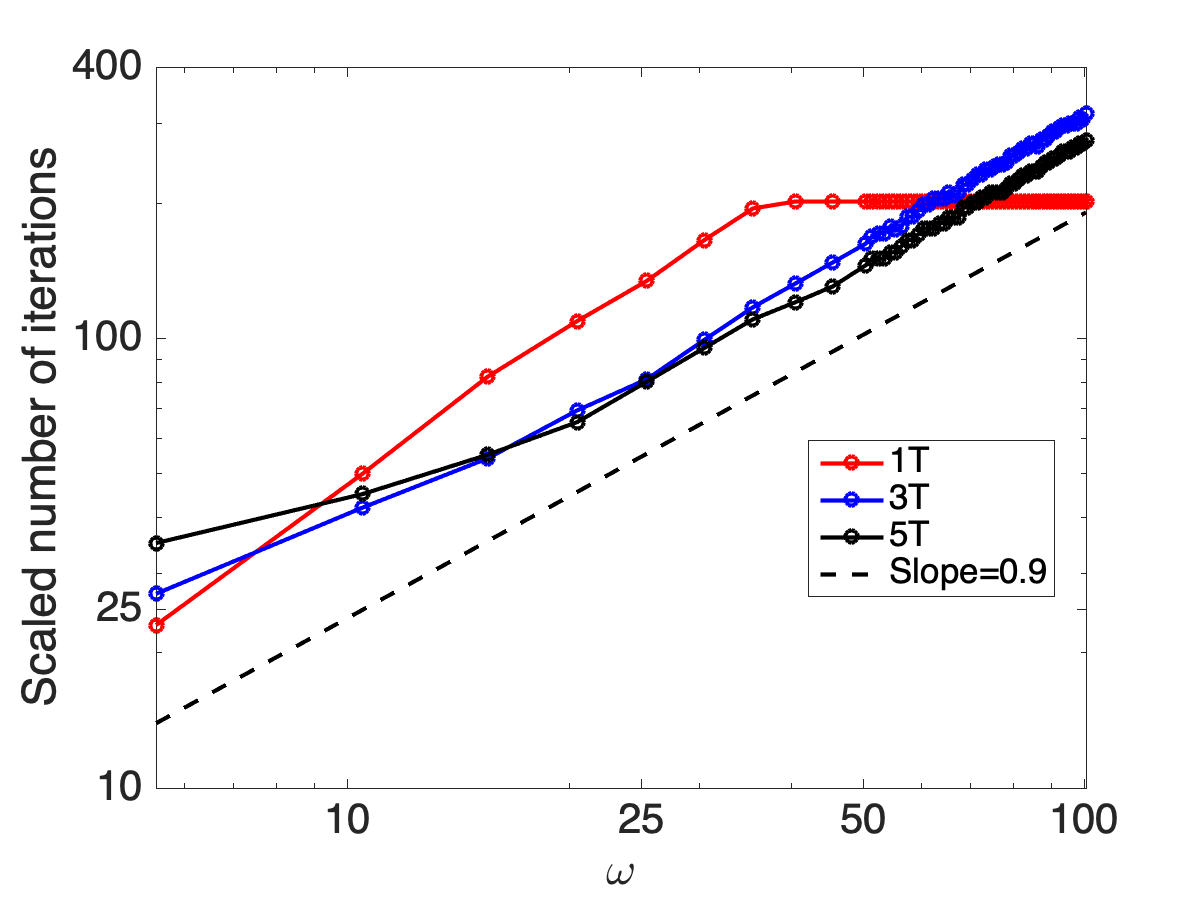}
\includegraphics[width=0.24\textwidth,clip]{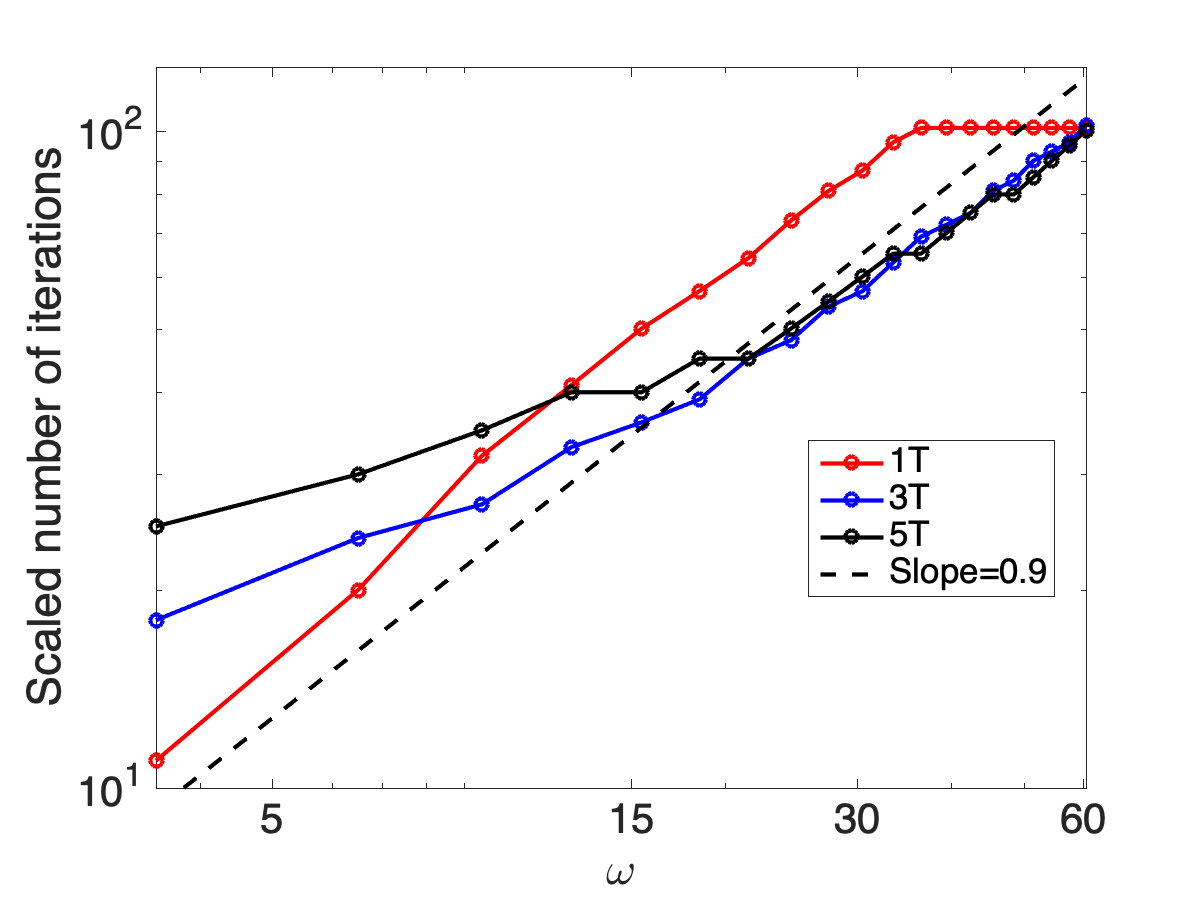}
\caption{Scaled number of iterations = $N\times$ number of iterations, if $T=N\frac{2\pi}{\omega}$. Scaled number of iterations as a function of frequency for different filtering time. Top: 2D open problem. Bottom: 3D open problem.  \label{fig:memory_lean_scaled}}
\end{figure}

\subsection{\pzc{SPD structure and condition numbers} for energy conserving 2D TM-model}
Here, we consider the 2D TM model with PEC boundary conditions, $\eps=\mu=1$, \pzc{$\omega=10$} and the Gaussian source
\begin{align}
J_z = -\omega \exp\left(-\omega^2( x^2+y^2)\right),
\end{align}
on the computational domain $[-1,1]^2$. We apply the Yee-EM-WaveHoltz method, and \pzc{filter over $10$ periods in the time-domain.} We use this example to verify that our method results in a SPD linear system. The code is implemented in a matrix free manner. The matrix $I-S$ is constructed column by column through the calculation of matrix-vector multiplication $(I-S)e_i$ where $e_i$ is a column vector whose $i$-th element is 1 and all other elements are 0.

 \pzc{We use $N=10,20,\dots,100$ elements in each direction.} We first compare $I-S$ and $(I-S)^T$. The $L_{\infty}$ norm of their difference is always on the level of machine accuracy. \pzc{The smallest eigenvalue of the resulting linear system is always positive. Moreover, when $N$ is large enough to resolve the wave structure, the condition number of the resulting matrix does not increase as $N$ grows. The condition numbers are $79.1305$ for $N=90$,  $75.5935$ for $N=95$ and $72.7603$ for $N=100$. This matches our previous observation that the total number of iterations does not grow with increased number of points per wavelength.}

\section{Conclusion}
In this paper, we proposed the EM-WaveHoltz method, which converts the frequency-domain problem into a time-domain problem with time periodic forcing. The main advantages of the proposed method are as follows.
\begin{enumerate}
\item The resulting linear system is positive definite, and the GMRES iterative solver converges reasonably fast even though  no preconditioning was used. 
\item The method is flexible and straightforward to implement, it only requires a time-domain solver. In this paper, we illustrated how either the classical Yee scheme or a discontinuous Galerkin method can be used to construct frequency-domain solvers.  
\item A unique feature of our EM-WaveHoltz method is that the solution to multiple frequencies can be obtained in a singe solve. 
\end{enumerate}
Potential future research directions are to design preconditioning strategies to further accelerate the convergence of the proposed iterative method. It would also be interesting to apply the method to more advanced dispersive material models.

\appendices
\section{Derivation of the EM-WaveHoltz iteration with $\cos$-forcing}\label{sec:cos_forcing}
The  real valued  $T = 2\pi/ \omega$-periodic solutions to \eqref{eq:maxwell_time_cos} is in the form:
\begin{align}
\WBE &= \hat{{\bf E}}_0 \cos(\omega t) + \hat{{\bf E}}_1 \sin(\omega t), \\
\WBH &= \hat{\bf{H}}_0 \cos (\omega t) + \hat{\bf{H}}_1 \sin (\omega t).
\end{align}
Matching the $\sin(\omega t)$ and $\cos(\omega t)$ term, we reach
\begin{align*}
\eps \omega (- \hat{{\bf E}}_0 )  &= \nabla\times  \hat{\bf{H}}_1 , \\
\eps \omega (\hat{{\bf E}}_1 )  &=  \nabla\times  \hat{\bf{H}}_0-\BJ , \\
\mu \omega (- \hat{{\bf H}}_0 )  &= - \nabla\times  \hat{\bf{E}}_1, \\
\mu \omega (\hat{{\bf H}}_1 )  &= -\nabla\times  \hat{\bf{E}}_0.
\end{align*}
Based on \eqref{eq:maxwell_frequency_re_im}, it follows that 
$$
\hat{\BE}_0 = \Re\{E\},\;\;\hat{\BH}_0 = \Re\{H\},\;\;\hat{\BE}_1 = -\Im\{E\},\;\;\hat{\BH}_1 = -\Im\{H\}.
$$
By construction, one can further verify that $\Pi (\Im\{\BE\},\Im\{\BH\})^T = (\Re\{\BE\},\Re\{\BH\})^T$.

\section{Analysis of energy conserving EM-WaveHoltz iteration}\label{sec:analysis_continuous}
Similar to \cite{appelo2020waveholtz}, we analyze the convergence of the simplified EM-WaveHoltz iteration for the energy conserving  case and show that $I-S$ is a self-adjoint positive definite operator.

Eliminating $\BH$ in the frequency-domain equation \eqref{eq:maxwell_frequency}, we have 
\begin{align}
\label{eq:TM-2nd}
-\omega^2\eps\BE= -\nabla\times\left(\frac{1}{\mu}\nabla \times\BE\right)-i\omega\BJ.
\end{align}
With the real-valued current source $J$, we further have
\begin{align}
\label{eq:TM-2nd2}
-\eps\omega^2\Im(\BE)= -\nabla\times\left(\frac{1}{\mu}\nabla \times\Im(\BE)\right)-\omega\BJ.
\end{align}

Eliminating $\WBH$ in the time-domain equation \eqref{eq:maxwell_time}, we obtain 
\begin{align}
\label{eq:TM-time-2nd}
\eps\partial_{tt}\WBE = -\nabla\times\left(\frac{1}{\mu}\nabla \times\WBE\right)-\omega\cos(\omega t) \BJ,
\end{align}
with $\WBE|_{t=0}=\Bnu_E$ and $\WBE_t = \Bzero$.

Suppose there is an  orthonormal basis of $L^2(\Omega)$ consisted by the eigenfunctions of the operator  $-\frac{1}{\eps}\nabla\times\left(\frac{1}{\mu}\nabla \times\right)$  (For example this holds under the assumptions of Theorem 8.2.4 in \cite{assous2018mathematical}).  Let the eigenfunctions $\{\phi_j\}_{j=1}^\infty$ consist an orthonormal basis of the $L^2$ space. Let $\{-\lambda_j^2\}_{j=1}^\infty$ denote the corresponding nonpositive eigenvalues. For simplicity of notations, we let $\Bnu_E=\Bnu$. Then, $\BE$, $\WBE$, $\BJ$, $\Bnu$ can be expanded as:
\begin{align}
\begin{split}
&\BE = \sum_{j=1}^\infty \BE_j\phi_j,\; \WBE = \sum_{j=1}^\infty \WBE_j\phi_j,\\
&\BJ = \sum_{j=1}^\infty \BJ_j\phi_j, \;\Bnu = \sum_{j=1}^\infty \Bnu_j\phi_j.
\end{split}
\end{align} 
Solve \eqref{eq:TM-2nd} and \eqref{eq:TM-time-2nd}:
\begin{align}
&\BE_j = \frac{\frac{1}{\eps}\omega \BJ_j}{\lambda_j^2-\omega^2},\\
&\WBE_j = \BE_j\left(\cos(\omega t)-\cos(\lambda_j t)\right)+\Bnu_j\cos(\lambda_j t).
\end{align}
Then, 
\begin{align}
\Pi\Bnu = \sum_{j=1}^\infty \Bbnu_j\phi_j, \; \bnu_j=(1-\beta(\lambda_j) )\BE_j + \beta(\lambda_j) \Bnu_j,
\label{eq:nu_bar}
\end{align}
where 
\begin{align}
\beta(\lambda) = \frac{2}{T}\int_{0}^T \left(\cos(\omega t)-\frac{1}{4}\right)\cos(\lambda t)dt.
\end{align}

Realizing that 
\begin{align}
\Pi \Bzero &= \sum_{j=1}^\infty\left((1-\beta(\lambda_j) )\BE_j + \beta(\lambda_j)\Bzero\right)\notag\\
&= \sum_{j=1}^\infty(1-\beta(\lambda_j) )\BE_j, \label{eq:pi_zero_result}
\end{align}
we have 
\begin{align}
\label{eq:S_expansion}
S\sum_{j=1}^\infty \Bnu_j\phi_j = \Pi\Bnu-\Pi \Bzero = \sum_{j=1}^\infty\beta(\lambda_j) \Bnu_j\phi_j.
\end{align}

Furthermore, as proved in \cite{appelo2020waveholtz}, the spectral radius $\rho$ of $S$ is strictly smaller than $1$:
\begin{align}
\rho \sim 1-6.33\delta^2,\quad \delta = \inf_j \frac{\lambda_j-\omega}{\omega}.
\end{align}
 As a result, when $\omega$ is not a resonance, 
\begin{align}
\lim_{n\rightarrow\infty} (\Pi \Bnu^n-\BE) = \lim_{n\rightarrow\infty}S^n (\Bnu^0-\BE) \rightarrow 0.
\end{align}
Furthermore, 
\begin{align}
\left( (I-S)\Bnu,\Bnu\right) \geq (1-\rho)||\Bnu||^2>0.
\end{align}
This also verifies that $I-S$ is positive definite. One can easily verify that $I-S$ is self-adjoint based on the expansion \eqref{eq:S_expansion}.

\section{Proof of Theorem \ref{thm:2d_tm}}\label{sec:analysis_discrete}

\begin{proof}[Proof of Theorem \ref{thm:2d_tm}]
Proof  of Theorem \ref{thm:2d_tm} is similar to the  proof of Theorem 2.4 of \cite{appelo2020waveholtz}. Here, we only point out the key steps. We expand all functions as
\begin{align}
\begin{split}
\WE_z^n = \sum_{j=1}^N (\WE_z)^n_j\psi_j,\; J_z = \sum_{j=1}^N (J_z)_j\psi_j,\\
(E_z) = \sum_{j=1}^N (E_z)_j\psi_j,\; \nu^\infty = \sum_{j=1}^N\nu_j^\infty\psi_j.
\end{split}
\end{align}
Then, 
\begin{align}
(E_z)_j = \frac{\frac{1}{\eps}\omega (J_z)_j}{\omega^2-\lambda_j^2},\quad \nu_j^\infty = \frac{\frac{1}{\eps}\omega (J_z)_j}{\womega^2-\lambda_j^2}.
\end{align}
Moreover, for $n\neq0$
\begin{align}
&(\WE_z)_j^{n+1}-2(\WE_z)_j^n+(\WE_z)_j^{n-1}+\Dt^2\lambda_j^2(\WE_z)^n_j\notag\\
=&-\omega\Dt^2\cos(\omega t^n) \frac{1}{\eps}J_z,
\end{align}
and
\begin{align}
(\WE_z)^0_j = \nu_j,\; (\WE_z)^1_j=(1-\frac{1}{2}\lambda_j^2\Dt^2)\nu_j-\frac{\omega}{2}\Dt^2(\frac{1}{\eps}J_z)_j.
\end{align}
Following Appenndix B of \cite{appelo2020waveholtz}, one can verify that 
\begin{align}
\label{eq:discrete_mode}
(\WE_z)^n_j = (\nu_j-\nu_j^\infty)\cos(\wlambda_jt^n)+\nu_j^\infty\cos(\omega t^n),
\end{align}
where
\begin{align}
\frac{\sin(\wlambda_j\Dt/2)}{\Dt/2} = \lambda_j.
\end{align}
Let $\Pi_h \nu = \sum_j \bar{\nu}_j\psi_j$. Then, one can obtain
\begin{align}
\bar{\nu}_j =\nu_j\beta_h(\wlambda_j)+(1-\beta_h(\wlambda_j) ),
\end{align}
where 
\begin{align}
\beta_h(\lambda)=\frac{2\Dt}{T}\sum_{n=0}^M\eta_n\cos(\lambda t^n)\left(\cos(\omega t^n)-\frac{1}{4}\right).
\end{align}
Lemma 2.5 of \cite{appelo2020waveholtz} shows that $|\beta_h(\wlambda_j)|\leq \rho_h:=\max(1-0.3\delta^2,0.63)$. Utilizing the fact that the composite trapezoidal rule is exact for 
pure periodic trigonometric functions of order less than the number of grid points, we complete the proof.
\end{proof}

\section{Verification of time error elimination in  Yee-EM-WaveHoltz}\label{appendix:time_error}
With the modification in \eqref{eq:omega_eliminate_time}, $\womega$  in Theorem \ref{thm:2d_tm} becomes 
$$\frac{\sin(\bar{\omega}\Dt/2)}{\Dt/2}=\omega.$$
Meanwhile, \eqref{eq:discrete_mode} becomes
\begin{align}
(\WE_z)^n_j = (\nu_j-\nu_j^\infty)\cos(\wlambda_jt^n)+\nu_j^\infty\cos(\bar{\omega} t^n),
\end{align}
and the original composite trapezoidal quadrature is no longer exact for $\nu_j^\infty\cos(\bar{\omega} t^n)$. Hence, the modification \eqref{eq:modified_quadrature} to the numerical quadrature is needed to eliminate time error from the numerical integration.


\ifCLASSOPTIONcaptionsoff
  \newpage
\fi



%

\bibliographystyle{IEEEtran}
\bibliography{reference}

\begin{IEEEbiographynophoto}{Daniel Appel\"{o}}
Bio:
Daniel Appelö holds a Ms degree in Electrical Engineering and a Ph. D. degree in Numerical Analysis from the Royal Institute of Technology in Sweden and is currently an Associate Professor in the Department of Computational Mathematics, Science and Engineering and the Department of Mathematics at Michigan State University. 
\end{IEEEbiographynophoto}

\begin{IEEEbiographynophoto}{Zhichao Peng}
Bio:
Zhichao Peng holds  a Ph. D. degree in Mathematics from the Rensselaer Polytechnic Institute in Troy, NY, USA in 2020. He is currently a research associate  in the Department of Mathematics at Michigan State University. 
\end{IEEEbiographynophoto}

\end{document}